\documentclass[reqno,oneside,12pt]{amsart}

\usepackage{hyperref}
\usepackage{longtable}
\usepackage[ansinew]{inputenc}
\usepackage{graphicx}
\usepackage{color}%
\usepackage[numbers, square]{natbib}
\usepackage{mathrsfs}
\usepackage{bbm}
\usepackage{tikz}
\usepackage{mathptmx}
\usepackage{amsmath,amsthm,amsfonts,amssymb}
\usepackage{dsfont,extarrows,enumerate}
\usepackage{fullpage}
\usepackage{xcolor}

\numberwithin{equation}{section}

\newcommand{\me}{\mathbb{E}}
\newcommand{\E}{\mathbb{E}}
\newcommand{\Q}{\mathbb{Q}}
\newcommand{\mr}{\mathbb{R}}
\newcommand{\mc}{\mathbb{C}}

\newcommand{\mn}{\mathbb{N}}
\newcommand{\mmp}{\mathbb{P}}
\renewcommand{\P}{\mathbb{P}}

\newcommand{\ms}{\mathfrak{s}}

\renewcommand{\Re}{{\rm Re}}
\renewcommand{\Im}{{\rm Im}}
\newcommand{\ii}{{\rm i}}
\newcommand{\Zeros}{{\bf Zeros}}

\DeclareMathOperator{\1}{\mathbbm{1}}
\newcommand{\eee}{{\rm e}}

\newtheorem{thm}{Theorem}[section]
\newtheorem{lemma}[thm]{Lemma}

\newtheorem{cor}[thm]{Corollary}

\newtheorem{assertion}[thm]{Proposition}
\theoremstyle{definition}

\theoremstyle{remark}
\newtheorem{rem}[thm]{Remark}

\begin{document}

\title{Limit theorems for random Dirichlet series}\date{}

\author{Dariusz Buraczewski}
\address{Dariusz Buraczewski, Mathematical Institute, University of Wroclaw, 50-384 Wroclaw, Poland}
\email{dariusz.buraczewski@math.uni.wroc.pl}

\author{Congzao Dong}
\address{Congzao Dong, School of Mathematics and Statistics, Xidian University, 710126 Xi'an, China}
\email{czdong@xidian.edu.cn}

\author{Alexander Iksanov}
\address{Alexander Iksanov, Faculty of Computer Science and Cybernetics, Taras Shev\-chen\-ko National University of Kyiv, 01601 Kyiv, Ukraine}
\email{iksan@univ.kiev.ua}

\author{Alexander Marynych}
\address{Alexander Marynych, Faculty of Computer Science and Cybernetics, Taras Shev\-chen\-ko National University of Kyiv, 01601 Kyiv, Ukraine}
\email{marynych@unicyb.kiev.ua}

\begin{abstract}
We prove a functional limit theorem in a space of analytic functions for the random Dirichlet series $D(\alpha;z)=\sum_{n\geq 2}(\log n)^{\alpha}(\eta_n+\ii \theta_n)/n^z$, properly scaled and normalized, where $(\eta_n,\theta_n)_{n\in\mathbb{N}}$ is a sequence of independent copies of a centered $\mathbb{R}^2$-valued random vector $(\eta,\theta)$ with a finite second moment and $\alpha>-1/2$ is a fixed real parameter. As a consequence, we show that the point processes of complex and real zeros of $D(\alpha;z)$ converge vaguely, thereby obtaining a universality result. In the real case, that is, when $\mathbb{P}\{\theta=0\}=1$, we also prove a law of the iterated logarithm for $D(\alpha;z)$, properly normalized, as $z\to (1/2)+$.
\end{abstract}



\maketitle

\section{Introduction}
Let $(X_n)_{n\in\mn}$ be a sequence of complex-valued random variables defined on some
probability space $(\Omega,\mathcal{F},\mathbb{P})$. Here, $\mn=\{1,2,\ldots\}$ is the set of positive integer numbers. The random series
$$
D(z):=\sum_{n\geq 1}\frac{X_n}{n^z},\quad z\in\mathbb{C}
$$
is called {\it random Dirichlet series}. For any
Dirichlet series, deterministic or random, there exist two real parameters $\sigma_a$ and $\sigma_c$ associated with the domains of convergence of $D:=(D(z))_{z\in\mathbb{C}}$, see Chapter 9 in \cite{Titchmarsh:1952}. The parameter $\sigma_a$ is called the {\it abscissa of absolute convergence} and the series $D(z)$ converges absolutely $\mathbb{P}$-almost surely (a.s.~in short) if $\Re(z)>\sigma_a$ and diverges absolutely a.s.~if $\Re(z)<\sigma_a$. Likewise, the parameter $\sigma_c$ is called the {\it abscissa of convergence} and the series $D(z)$ converges a.s.~if $\Re(z)>\sigma_c$ and diverges  a.s.~if $\Re(z)<\sigma_c$. The open half-planes
$$
H_{\sigma_a}=\{z\in\mathbb{C}:\Re(z)>\sigma_a\}\quad\text{and}\quad H_{\sigma_c}=\{z\in\mathbb{C}:\Re(z)>\sigma_c\}
$$
are called the {\it half-plane of absolute convergence} and the {\it half-plane of convergence}, respectively. The function $D$ is a.s. ~analytic in $H_{\sigma_c}$ as an a.s.~uniformly convergent series of analytic  functions, see p.~291 in \cite{Titchmarsh:1952}.

The values $\sigma_a$ and $\sigma_c$ depend heavily upon the properties of the random coefficients $(X_n)_{n\in\mn}$.
According to Section 9.14 in \cite{Titchmarsh:1952},
\begin{equation}\label{eq:sigma_c_formula}
\sigma_c=\limsup_{n\to\infty}\frac{\log |X_1+\ldots+X_n|}{\log n}\quad\text{a.s.}
\end{equation}
If the coefficients $X_1$, $X_2,\ldots$ are independent, then, by a zero-one law, $\sigma_c\in [-\infty, +\infty]$ is a degenerate random variable. Otherwise, $\sigma_c$ may be nondegenerate.

In the present paper we focus on a particular instance of random Dirichlet series, in which $X_n:=(\log n)^{\alpha}(\eta_n+\ii \theta_n)$, $n\in\mn$, where
$\alpha\in\mathbb{R}$ is a fixed parameter and $(\eta_n,\theta_n)_{n\in\mn}$ are independent copies of an $\mr^2$-valued random vector $(\eta,\theta)$ satisfying
\begin{equation}\label{eq:main_assumptions}
\me \eta=\me\theta= 0\quad\text{and}\quad 0<\me\eta^2+\me\theta^2<\infty.
\end{equation}
Here and in what follows, $\me$ denotes the expectation with respect to the probability measure $\mmp$. The corresponding random Dirichlet series will be denoted by $D_\alpha:=(D(\alpha;z))_{z\in\mathbb{C}}$, that is,
$$
D(\alpha;z):=\sum_{n\geq 2}\frac{(\log n)^{\alpha}(\eta_n+\ii \theta_n)}{n^{z}},\quad z\in\mathbb{C}.
$$
For each fixed $\alpha\in\mathbb{R}$, a specialization of ~\eqref{eq:sigma_c_formula} to $D_\alpha$ yields
\begin{equation}\label{eq:sigma_c_under_main_assumption}
\mathbb{P}\left\{\sigma_c=\frac{1}{2}\right\}=1.
\end{equation}
This follows from the law of the iterated logarithm for weighted sums of independent identically distributed random variables, see Theorem 3 in \cite{Teicher:1974}. Indeed, this theorem entails
$$
\limsup_{n\to\infty}\left(2n(\log n)^{2\alpha}\log\log n\right)^{-1/2}\left|\sum_{k=2}^{n}(\log k)^{\alpha}(\eta_k+\ii \theta_k)\right|\in (0,\infty)\quad\text{a.s.},
$$
thereby justifying ~\eqref{eq:sigma_c_under_main_assumption}. Thus, for every fixed $\alpha\in\mathbb{R}$, $z\mapsto D(\alpha;1/2+z)$ is a random analytic function in the right open half-plane
$$
H_{0}=\{z\in\mathbb{C}:\Re(z)>0\}.
$$
The analytic character of $D_\alpha$ on the critical line $\{z\in\mc:\Re(z)=1/2\}$ is discussed in
Theorem 4 on p.~44 in \cite{Kahane:1968}. Further results of this flavor for general random Dirichlet series can be found in \cite{Ding+Xiao:2006}. However, we do not pursue this line of research in the present paper.

Limit behavior of $D_0$ has received some attention in the recent years. The main motivation for our work comes from the two recent papers \cite{Aymone:2019} and \cite{Aymone+Frometa+Misturini:2020}, where the particular case
\begin{equation}\label{eq:simplest_case}
\mathbb{P}\{\eta=\pm 1\}=1/2\quad\text{and}\quad \mathbb{P}\{\theta=0\}=1
\end{equation}
was investigated. In particular, under ~\eqref{eq:simplest_case}, a one-dimensional central limit theorem and a law of the iterated logarithm for $D_0$, properly normalized, can be found in~\cite{Aymone+Frometa+Misturini:2020}. We show that both results continue to hold under more general assumption~\eqref{eq:main_assumptions} for arbitrary $\alpha>-1/2$. Furthermore, we upgrade the one-dimensional central limit theorem to a functional limit theorem in an appropriate space of analytic functions. By a standard reasoning, the latter entails weak convergence of the random point process of zeros of $D_{\alpha}$. Example 5.5 in \cite{Shirai:2012} provides a functional limit theorem, along with a limit theorem for the zeros, for a counterpart of $D_0$, in which the summation extends over prime indices and $\eta+\ii \theta$ has a uniform distribution on $\{z\in\mathbb{C}: |z|=1\}$.

Throughout the paper we write $\overset{\mmp}{\to}$ to denote convergence in probability, and $\Longrightarrow$, ${\overset{{\rm
d}}\longrightarrow}$ and ${\overset{{\rm f.d.d.}}\longrightarrow}$ to denote weak convergence in a function space, weak convergence
of one-dimensional and finite-dimensional distributions, respectively. We also identify $\mr^2$ and $\mc$ via the canonical isomorphism and consider $\mr^2$-valued processes as $\mc$-valued and vice versa.

\section{Functional limit theorem for the random Dirichlet series and convergence of its zeros}
As expected, in the limit theorems for $(D(\alpha;z))_{z\in H_0}$ we shall encounter moments and covariance of $(\eta,\theta)$. As a preparation, put
$$
\mathfrak{C}:=\left(\begin{matrix}{\rm Var}\, \eta & {\rm Cov}\, (\eta,\theta)\\
{\rm Cov}\, (\eta,\theta) & {\rm Var}\, \theta\\ \end{matrix}\right)=:\left(\begin{matrix}\sigma_1^2 & \rho\\
\rho & \sigma_2^2 \end{matrix}\right).
$$
In what follows, $\lfloor x\rfloor$ denotes the integer part of $x\in\mr$ and $y^\intercal$ denotes the transpose of a column vector $y$.

Let $(B(t))_{t\geq 0}=(B_1(t),B_2(t))^{\intercal}_{t\geq 0}$ be a standard two-dimensional centered Brownian motion with the independent components $(B_1(t))_{t\geq 0}$ and $(B_2(t))_{t\geq 0}$. The classical invariance principle tells us that, under ~\eqref{eq:main_assumptions},
\begin{equation}\label{eq:donsker}
\left(\frac{\eta_1+\eta_2+\ldots+\eta_{\lfloor nt\rfloor}}{\sqrt{n}},\frac{\theta_1+\theta_2+\ldots+\theta_{\lfloor nt\rfloor}}{\sqrt{n}}\right)^{\intercal}_{t\geq 0}~\Longrightarrow~(\mathfrak{C}^{1/2} B(t))_{t\geq 0},\quad n\to\infty
\end{equation}
on the Skorokhod space $D([0,\infty),\mr^2)$ of $\mr^2$-valued c\`{a}dl\`{a}g functions defined on $[0,\infty)$, endowed with the $J_1$-topology. As a consequence,
\begin{equation*}
\left(\frac{(\eta_1+\ii\theta_1)+(\eta_2+\ii\theta_2)+\ldots+(\eta_{\lfloor nt\rfloor}+\ii \theta_{\lfloor nt\rfloor})}{\sqrt{n}}\right)_{t\geq 0}~\Longrightarrow~((1,\ii)\mathfrak{C}^{1/2} B(t))_{t\geq 0},\quad n\to\infty
\end{equation*}
on the Skorokhod space $D([0,\infty),\mc)$. Here and hereafter, for $a,b,c,d\in \mc$, $(a,b)(c,d)=ac+bd$. Observe that the weak limit in \eqref{eq:donsker} is a two-dimensional Wiener process with the covariance matrix $\mathfrak{C}$.

Now we introduce a stochastic process which serves as a weak limit of $(D(\alpha; 1/2+sz))_{z\in H_0}$, properly normalized, as $s\to 0+$.

\subsection{The limit process}

For $z\in H_0$, $\alpha>-1/2$ and $j=1,2$, define the Skorokhod integral
$$
\mathcal{I}_j(\alpha;z):=\int_{[0,\infty)}y^{\alpha}\eee^{-zy}{\rm d}B_j(y).
$$
Alternatively, $\mathcal{I}_j(\alpha;z)$ can be thought of as the result of integration by parts
$$
\mathcal{I}_j(\alpha;z)=-\int_{[0,\infty)}B_j(y){\rm d}(y^{\alpha}\eee^{-zy}).
$$
A $\mc$-valued process $\mathcal{I}_\alpha:=(\mathcal{I}(\alpha;z))_{z\in H_0}$ is then defined by the product
$$
\mathcal{I}(\alpha;z)
:= (1,\ii) \mathfrak{C}^{1/2}\left(
\begin{matrix}
\mathcal{I}_1(\alpha;z)\\
\mathcal{I}_2(\alpha;z)
\end{matrix}
\right).
$$
Note that for $\alpha\leq -1/2$ the integral defining $\mathcal{I}_j(\alpha;z)$ diverges due to a singularity at $0$. According to Theorem 5a in \cite{Widder:1946}, the functions $(\mathcal{I}_j(\alpha;z))_{z\in H_0}$, $j=1,2$ are a.s.~analytic in $H_0$, for every fixed $\alpha>-1/2$. Thus, $\mathcal{I}_\alpha$ is also analytic in $H_0$ as a linear combination of analytic functions. Summarizing we conclude that the random process $\mathcal{I}_\alpha$ is a centered Gaussian analytic function on $H_0$. Its covariance structure is given in the next proposition which can be checked by direct calculations using independence of $(\mathcal{I}_1(\alpha;z))$ and $(\mathcal{I}_2(\alpha;z))$.

\begin{assertion}
The covariances of the process $(\mathcal{I}(\alpha;z))_{z\in H_0}$ are given by
\begin{equation}\label{limit1}
\me \left(\mathcal{I}(\alpha;z_1)\mathcal{I}(\alpha;z_2)\right)=\frac{\Gamma(1+2\alpha)(\sigma_1^2-\sigma_2^2+2\ii \rho)}{(z_1+z_2)^{1+2\alpha}},\quad z_1,z_2\in H_0
\end{equation}
and
\begin{equation}\label{limit2}
\me \left(\mathcal{I}(\alpha;z_1)\overline{\mathcal{I}(\alpha;z_2)}\right)= \frac{\Gamma(1+2\alpha)(\sigma_1^2+\sigma_2^2)}{(z_1+\overline{z_2})^{1+2\alpha}},\quad z_1,z_2\in H_0,
\end{equation}
where $\Gamma$ is the Euler gamma function and $\overline{x}$ denotes the complex conjugate of $x\in \mc$.
\end{assertion}

A role of the real parameter $\alpha$ in the definitions of $D_\alpha$ and $\mathcal{I}_\alpha$ is revealed by the following observation. Let $\mathcal{D}_{-}^{\alpha}$ be a fractional derivative operator defined by
\begin{equation}\label{rem:frac_derivative}
(\mathcal{D}_{-}^{\alpha}(f))(z)=(-1)^m \left(\frac{{\rm d}}{{\rm d}z}\right)^{m}(I_{-}^{m-{\alpha}}(f))(z),\quad m:=\lfloor \alpha\rfloor+1,
\end{equation}
where
$$
(I_{-}^{\gamma}(f))(z)=\frac{1}{\Gamma(\gamma)}\int_0^{\infty}f(z+u)u^{\gamma-1}{\rm d}u,\quad \gamma\in (0,1],
$$
see~(22.18) and~(22.21) in \cite{Samko+Kilbas+Marichev:1993}. Then
\begin{equation}\label{eq:frac_dreivative}
D(\alpha;\cdot)=\mathcal{D}_{-}^{\alpha}(D(0;\cdot)), \quad \mathcal{I}(\alpha;\cdot)=\mathcal{D}_{-}^{\alpha}(\mathcal{I}(0;\cdot)).
\end{equation}
Thus, for instance, passing from $\alpha$ to $\alpha+1$ amounts to taking the usual derivative. Equations~\eqref{eq:frac_dreivative} also highlight the special role of the value $\alpha=0$ both for the Dirichlet series $D_\alpha$ and for the process $\mathcal{I}_\alpha$. It turns out that for the two particular choices of the matrix $\mathfrak{C}$, the limit process $\mathcal{I}_\alpha$ and its various time-changed versions pop up frequently in modern probability, with the case $\alpha=0$ being of special importance.

The first of the aforementioned choices is $\mathfrak{C}=I$, where $I$ is the identity matrix. Then
\begin{equation}\label{eq:C=I_limit_process}
\mathcal{I}(\alpha;z)=\int_{[0,\infty)}y^{\alpha}\eee^{-zy}{\rm d}B_{\mc}(y),\quad z\in H_0,
\end{equation}
where $B_{\mc}$ is the standard {\it complex} Brownian motion. An exponential change of time $z\mapsto \eee^{2z}$ which maps the union of parallel horizontal strips
$$
\mathcal{S}_0:=\bigcup_{k\in\mathbb{Z}}\left\{z\in\mc : \Im(z)\in (-\pi/4+\pi k,\pi/4+\pi k)\right\}
$$
onto $H_0$ leads to a Gaussian process
$$
S_{\alpha}(z):=2^\alpha (\Gamma(1+2\alpha))^{-1/2}\eee^{(1+2\alpha)z}\mathcal{I}(\alpha; \eee^{2z}),\quad  z\in \mathcal{S}_0,
$$
with the covariance structure
$$
\E (S_{\alpha}(z_1)S_{\alpha}(z_2))=0\quad\text{and}\quad\E (S_{\alpha}(z_1)\overline{S_{\alpha}(z_2)})=\frac{1}{(\cosh(z_1-\overline{z_2}))^{1+2\alpha}}.
$$
In particular, this means that the process $(S_\alpha(t))_{t\in\mathbb{R}}$ is a complex-valued stationary Gaussian process on $\mr$. Setting $\alpha=0$, replacing $\mathcal{I}(0; \eee^{2z})$ with $\mathcal{I}_1(0; \eee^{2z})$ in the definition of $S_\alpha$ and only considering $z\in\mr$ we obtain a centered real-valued stationary Gaussian process $S_0$, which was recently investigated in ~\cite{Kabluchko:2019}.
Another important time-change is constructed as follows.  Let $\mathbb{D}:=\{z\in\mc:|z| < 1\}$ be the open unit disk in $\mc$. The linear fractional transformation
\begin{equation}\label{eq:lin_frac}
\varphi(z)=\frac{1+z}{1-z},\quad z\in\mathbb{D}
\end{equation}
maps $\mathbb{D}$ conformally to $H_0$. A time-changed process defined by the transformation
\begin{equation}\label{eq:connection_with_random_power_series1}
f_{\alpha,\mathbb{C}}(z):=\frac{2^{\alpha}(\Gamma(1+2\alpha))^{-1/2}}{(1-z)^{1+2\alpha}}\mathcal{I}(\alpha;\varphi(z))=\frac{2^{\alpha}(\Gamma(1+2\alpha))^{-1/2}}{(1-z)^{1+2\alpha}}\mathcal{I}\left(\alpha;\frac{1+z}{1-z}\right),\quad z \in\mathbb{D}
\end{equation}
is a Gaussian process with the covariance structure
$$
\E (f_{\alpha,\mathbb{C}}(z_1)f_{\alpha,\mathbb{C}}(z_2))=0\quad\text{and}\quad \E (f_{\alpha,\mathbb{C}}(z_1)\overline{f_{\alpha,\mathbb{C}}(z_2)})=\frac{1}{(1-z_1\overline{z_2})^{1+2\alpha}},\quad z_1,z_2\in\mathbb{D}.
$$
The latter formula implies that $f_{\alpha, \mathbb{C}}$ has the same distribution as a random Gaussian power series
\begin{equation}\label{eq:power_series_complex}
\sum_{n=0}^{\infty}\mathcal{N}^{\mc}_n \frac{\sqrt{(1+2\alpha)(2+2\alpha)\cdots(n+2\alpha)}}{\sqrt{n!}}z^n,\quad z\in \mathbb{D},
\end{equation}
where $(\mathcal{N}^{\mc}_n)_{n\geq 0}$ is a sequence of independent standard complex Gaussian variables. Random series~\eqref{eq:power_series_complex} is known in the literature as {\it hyperbolic Gaussian analytic function} and its properties have been much studied, with many
basic aspects already covered in Chapter 13 of ~\cite{Kahane:1968}.  An important feature of $f_{\alpha, \mathbb{C}}$
is that the point process of its complex zeros is determinantal if, and only if, $\alpha=0$, see \cite{Peres+Virag:2005} and Chapter 5 in \cite{Hough+Krishnapur+Peres+Virag}.

The second important choice of the matrix $\mathfrak{C}$ corresponds to the real case, when $\mathbb{P}\{\theta=0\}=1$ and $\alpha=0$. Then
$\mathfrak{C}$ has a unique non-zero entry $\sigma_1^2$ and $$
\mathcal{I}(0;z)=\int_{[0,\infty)}\eee^{-zy}{\rm d}B_1(y),\quad z\in H_0.
$$
We assume, without loss of generality, that $\sigma_1^2=1$. A time-changed process
\begin{equation}\label{eq:connection_with_random_power_series2}
f_{\mathbb{R}}(z):=\frac{1}{1-z}\mathcal{I}\left(0;\frac{1+z}{1-z}\right),\quad z\in\mathbb{D}
\end{equation}
has the same distribution as a random power series
\begin{equation}\label{eq:power_series_real}
\sum_{n=0}^{\infty}\mathcal{N}^{\mr}_n z^n,\quad z\in\mathbb{D},
\end{equation}
where $(\mathcal{N}^{\mr}_n)_{n\geq 0}$ is a sequence of independent standard real Gaussian variables. The process $f_{\mathbb{R}}$ has also been much studied. In particular, the point processes of both real and complex zeros of $f_{\mathbb{R}}$ are known to be Pffafian, see \cite{Matsumoto+Shirai:2013}.

\subsection{Distributional limit theorems}

Let $\mathcal{A}(H_{0})$ be the space of analytic functions on $H_{0}$,	endowed with the topology of uniform convergence on compact subsets of $H_{0}$. Our first result is a functional limit theorem in the space $\mathcal{A}(H_{0})$ for the scaled processes $(D(\alpha; 1/2+sz))_{z\in H_{0}}$, when a real parameter $s$ tends to $0+$.
\begin{thm}\label{thm:flt_main}
Assume~\eqref{eq:main_assumptions} and let $\alpha>-1/2$. Then the following weak convergence holds in the space of probability measures on $\mathcal{A}(H_{0})$:
$$
\left(s^{1/2+\alpha}D(\alpha;1/2+sz)\right)_{z\in H_0}~\Longrightarrow~(\mathcal{I}(\alpha;z))_{z\in H_0},\quad s\to 0+.
$$
\end{thm}

Note that if $(f(z))_{z\in H_0}\in\mathcal{A}(H_0)$, then $(f(x))_{x>0}\in C((0,\infty),\mc)$, where $C((0,\infty),\mc)$ is the space of $\mathbb{C}$-valued continuous functions defined on $(0,\infty)$ and endowed with the topology of locally uniform convergence. Theorem~\ref{thm:flt_main} immediately implies the following statement.
\begin{cor}\label{thm:flt_main_cor}
Assume~\eqref{eq:main_assumptions} and let
$\alpha>-1/2$. Then the following weak convergence holds in the space of probability measures on $C((0,\infty),\mc)$:
$$
\left(s^{1/2+\alpha}D(\alpha;1/2+sx)\right)_{x>0}~\Longrightarrow~(\mathcal{I}(\alpha;x))_{x>0},\quad s\to 0+.
$$
\end{cor}

In the real case when $\mathbb{P}\{\theta=0\}=1$, Corollary~\ref{thm:flt_main_cor} amounts to weak convergence of probability measures on $C((0,\infty),\mr)$ the space of real-valued continuous functions defined on $(0,\infty)$ endowed with the topology of locally uniform convergence. With a view towards a law of the iterated logarithm (Theorem \ref{main}) we only formulate a one-dimensional central limit theorem in this setting. As a preparation, using \eqref{limit1} with $z_1=z_2=1$ we conclude that $\mathcal{I}(\alpha;1)=\sigma_1\mathcal{I}_1(\alpha;1)$ has the same distribution as $(2^{-1-2\alpha} \Gamma(1+2\alpha)\sigma_1^2)^{1/2} \cdot {\rm Normal} (0,1)$, where ${\rm Normal} (0,1)$ denotes a random variable with the standard normal distribution. With this at hand, putting in Corollary \ref{thm:flt_main_cor} $x=1$ we arrive at the following result.
\begin{cor}\label{thm:flt_main_cor_real}
Assume that $\me\eta=0$, $\sigma_1^2=\me \eta^2\in (0,\infty)$ and let $\alpha>-1/2$. Then
$$
\Big(\frac{(2s)^{1+2\alpha}}{\Gamma(1+2\alpha)\sigma_1^2}\Big)^{1/2} \sum_{k\geq 2} \frac{(\log k)^\alpha}{k^{1/2+s}} \eta_k ~{\overset{{\rm d}}\longrightarrow}~ {\rm Normal} (0,1),\quad s\to 0+.
$$
\end{cor}

\subsection{Convergence of zeros}

Given a locally compact metric space $\mathbb{X}$, denote by $M_p(\mathbb{X})$ the space of locally finite point measures on $\mathbb{X}$ endowed with the vague topology. A random element with values in $M_p(\mathbb{X})$ is called a {\it random point measure} on $\mathbb{X}$. For a function $f:\mc\to\mc$, which is analytic in a domain $\Lambda\subset\mc$ and does not vanish identically, denote by $\Zeros_{\Lambda}(f)$ the locally finite point measure on $\Lambda$ counting the zeros of $f$ in $\Lambda$ with multiplicities.

A direct consequence of Theorem~\ref{thm:flt_main} is the following limit theorem for the point process of zeros of $(D(\alpha;1/2+z))_{z\in H_0}$, see Lemma 4.3 and Remark 4.4 in \cite{Kabluchko+Klimovsky:2014}.
\begin{assertion}\label{prop:complex_zeros}
Under the assumptions of Theorem~\ref{thm:flt_main}, the following weak convergence holds in the space of probability measures on $M_p(H_0)$:
$$
\Zeros_{H_0}(D(\alpha; 1/2+s(\cdot)))~\Longrightarrow~\Zeros_{H_0}(\mathcal{I}(\alpha;\cdot)),\quad s\to 0+.
$$
\end{assertion}

Recall that $\varphi$ denotes the linear fractional transformation (see ~\eqref{eq:lin_frac}), which maps the open unit disk $\mathbb{D}$ onto $H_0$. Proposition~\ref{prop:complex_zeros} in combination with representation~\eqref{eq:connection_with_random_power_series1} entails the following result.
\begin{cor}
Assume that $\E\eta^2=\E\theta^2$, $\rho=0$ and let $\alpha>-1/2$. Then
$$
\Zeros_{H_0}(D(\alpha;1/2+s(\cdot)))~\Longrightarrow~\varphi\left(\Zeros_{\mathbb{D}}(f_{\alpha,\mathbb{C}}(\cdot))\right),\quad s\to 0+,
$$
where $f_{\alpha,\mathbb{C}}$ is the hyperbolic Gaussian analytic function defined by~\eqref{eq:power_series_complex}.  If $\alpha=0$, then the point process $\Zeros_{\mathbb{D}}(f_{0,\mathbb{C}}(\cdot))$ is a determinantal point process with the joint intensity as given in Theorem 1 of \cite{Peres+Virag:2005}.
\end{cor}

Using the fact that, in the special case $\alpha=0$, the point process $\Zeros_{\mathbb{D}}(f_{0,\mathbb{C}}(\cdot))$ is determinantal, one can
deduce further information on the number of zeros in special disks.
Fix $r\in (0,1)$ and note that the linear fractional transformation $\varphi$ maps the open disk $\mathbb{D}(r):=\{z:|z|<r\}$ bijectively to the open disk
\begin{equation}\label{eq:tilde_disk}
\widetilde{\mathbb{D}}(r):=\left\{z\in\mathbb{C}:\left|z-\frac{1+r^2}{1-r^2}\right|<\frac{2r}{	1-r^2}\right\}\subset H_0.
\end{equation}
This can be proved by checking that $|\varphi^{-1}(z)|<r$ if, and only if, $z\in\widetilde{\mathbb{D}}(r)$. Furthermore, if $0<r_1<r_2<1$, then $\widetilde{\mathbb{D}}(r_1)\subset \widetilde{\mathbb{D}}(r_2)$ and $\widetilde{\mathbb{D}}(r)\uparrow H_0$, as $r\uparrow 1$. Also, for every $r\in (0,1)$,
$$
\varphi\left(\Zeros_{\mathbb{D}}(f_{0,\mathbb{C}}(\cdot))\right)(
\widetilde{\mathbb{D}}(r))
=\Zeros_{\mathbb{D}}(f_{0,\mathbb{C}}(\cdot))(
\varphi^{-1}(\widetilde{\mathbb{D}}(r)))
=\Zeros_{\mathbb{D}}(f_{0,\mathbb{C}}(\cdot))(
\mathbb{D}(r))=:N_r,
$$
that is, $N_r$ is the number of zeros of $f_{0,\mathbb{C}}$ lying inside
$\mathbb{D}(r)$. According to Corollary 5.1.7 in \cite{Hough+Krishnapur+Peres+Virag},
\begin{equation}\label{eq:number_of_zeros_gen_func}
\mathbb{E} (1+t)^{N_r}=\prod_{k\geq 1}(1+r^{2k}t),\quad t\in\mr.
\end{equation}
Further properties of the random variable $N_r$ can be found in Corollary 5.1.8 in \cite{Hough+Krishnapur+Peres+Virag}.

Putting things together we obtain
\begin{cor}
Assume that $\E\eta^2=\E\theta^2$ and $\rho=0$. Let $r\in (0,1)$ be fixed and denote by $N_r(s)$ the number of zeros of $z\mapsto D(0;1/2+sz)$ in the disk $\widetilde{\mathbb{D}}(r)$ defined by~\eqref{eq:tilde_disk}. Then
$$
N_r(s)~{\overset{{\rm d}}\longrightarrow}~N_r,\quad s\to 0+,
$$
where $N_r$ is a random variable with the generating function given by~\eqref{eq:number_of_zeros_gen_func}.
\end{cor}

In the real case when $\mathbb{P}\{\theta=0\}=1$ it is more natural to consider the point process of real zeros. A counterpart of Proposition~\ref{prop:complex_zeros} for the real zeros is given below. For $a,b\in \mr\cup\{\pm \infty\}$, $a<b$, let $\Zeros_{(a,b)}(f)$ denote the locally finite point measure on $(a,b)$ counting the real zeros (with multiplicities) of a function $f$ analytic in some domain of $\mc$  containing $(a,b)$.
\begin{assertion}\label{prop:real_zeros}
Assume that $\mathbb{P}\{\theta=0\}=1$, $\me\eta=0$, $\sigma_1^2=\me \eta^2\in (0,\infty)$ and let $\alpha>-1/2$. Then the following weak convergence holds in the space of probability measures on $M_p((0,\infty))$:
\begin{equation}\label{eq:real_zeros_pp_conv}
\Zeros_{(0,\infty)}(D(\alpha; 1/2+s(\cdot)))~\Longrightarrow~\Zeros_{(0,\infty)}(\mathcal{I}(\alpha; \cdot)),\quad s\to 0+.
\end{equation}
\end{assertion}

Note that the function $\varphi$ maps bijectively the open interval $(-1,1)$ to the positive half-line $(0,\infty)$. In the case $\alpha=0$ the following holds true in view of~\eqref{eq:connection_with_random_power_series2}.  \begin{cor}
In the setting of Proposition~\ref{prop:real_zeros},
\begin{equation*}
\Zeros_{(0,\infty)}(D(0; 1/2+s(\cdot)))~\Longrightarrow~\varphi(\Zeros_{(-1,1)}(f_{\mr}(\cdot))),\quad s\to 0+,
\end{equation*}
where $f_{\mr}$ is the random power series defined by~\eqref{eq:power_series_real}. The point process $\Zeros_{(-1,1)}(f_{\mathbb{R}}(\cdot))$ is a Pfaffian point process with the joint intensity described by Theorem 2.1 in \cite{Matsumoto+Shirai:2013}.
\end{cor}

An important feature of random Dirichlet series $D(\alpha;\cdot)$ revealed by Propositions~\ref{prop:complex_zeros} and~\ref{prop:real_zeros} is that the distributions of the limit point processes of zeros only depend on the covariance structure of $(\eta,\theta)$ and do not depend on the distribution of $(\eta,\theta)$. Such a phenomenon is usually referred to as {\it local universality} and has already been observed for many models, see, for instance, \cite{Angst+Pham+Poly:2018, Angst+Poly:2022, Iksanov+Kabluchko+Marynych:2016, Kabluchko+Zaporozhets:2014, Tao+Vu:2015}.

\section{A law of the iterated logarithm}

Our second main result given in Theorem \ref{main} is a law of the iterated logarithm. As usual, a hint concerning the form of this law is given by the central limit theorem, Corollary \ref{thm:flt_main_cor_real}.

For a family $(x_t)$ we denote by $C((x_t))$ the set of its limit points.
\begin{thm}\label{main}
Assume that $\mathbb{P}\{\theta=0\}=1$, $\me\eta=0$, $\sigma_1^2:=\me \eta^2\in (0,\infty)$ and let $\alpha>-1/2$. Then
\begin{equation}\label{auxlimsup}
{\lim\sup}_{s\to 0+} \Big(\frac{s^{1+2\alpha}}{\log\log 1/s }\Big)^{1/2}\sum_{k\geq 2}\frac{(\log k)^\alpha}{k^{1/2+s}} \eta_k= \Big(\frac{\sigma_1^2 \Gamma(1+2\alpha)}{2^{2\alpha}}\Big)^{1/2} \quad\text{{\rm a.s.}}
\end{equation}
and
\begin{equation}\label{auxliminf}
{\lim \inf}_{s\to 0+}\Big(\frac{s^{1+2\alpha}}{\log\log 1/s }\Big)^{1/2}\sum_{k\geq 2}\frac{(\log k)^\alpha}{k^{1/2+s}} \eta_k=-\Big(\frac{\sigma_1^2 \Gamma(1+2\alpha)}{2^{2\alpha}}\Big)^{1/2} \quad\text{{\rm a.s.}}
\end{equation}
In particular,
\begin{equation}\label{eq:limset}
C\bigg(\bigg(\Big(\frac{2^{2\alpha}}{\sigma_1^2 \Gamma(1+2\alpha)}\frac{s^{1+2\alpha}}{\log\log 1/s}\Big)^{1/2}\sum_{k\geq 2}\frac{(\log k)^\alpha}{k^{1/2+s}}\eta_k: s\in (0, 1/\eee)\bigg)\bigg)=[-1,1]\quad\text{{\rm a.s.}}
\end{equation}
\end{thm}

\begin{rem}In Theorem 1.1 of \cite{Aymone+Frometa+Misturini:2020} the following law of the iterated logarithm was proved under ~\eqref{eq:simplest_case}:
\begin{equation}\label{eq:motiv}
{\lim\sup}_{s\to 0+} \Big(\frac{s}{\log\log 1/s}\Big)^{1/2}\sum_{k\geq 1}\frac{1}{k^{1/2+s}} \eta_k= 1.
\end{equation}
The present article has partly been motivated by our desire to extend \eqref{eq:motiv} to centered i.i.d.\ random variables with finite second moment and more general weights. In particular, here, finiteness of the exponential moments of $\eta$ is not assumed which leads to additional technical complications.
\end{rem}

\section{Proof of Theorem~\ref{thm:flt_main} and Proposition~\ref{prop:real_zeros}}

\subsection{Proof of Theorem~\ref{thm:flt_main}}

We first prove weak convergence of the finite-dimensional distributions and then check tightness in the space $\mathcal{A}(H_0)$.

For every fixed $z\in H_0$ and $s>0$, the variable
$$
s^{1/2+\alpha}D(\alpha; 1/2+sz)=s^{1/2+\alpha}\sum_{k\geq 2}\frac{(\log k)^{\alpha}}{k^{1/2+sz}}(\eta_k+{\rm i}\theta_k)
$$
is an (infinite) sum of centered independent random variables with finite second moments. Thus, to check weak convergence of the finite-dimensional distributions it is enough to check convergence of covariances and then the Lindeberg-Feller condition.

\vspace{2mm}
\noindent
{\sc Convergence of covariances.} For fixed $z_1,z_2\in H_0$,
\begin{align*}
&\hspace{-3cm} s^{1+2\alpha}\me \left(D(\alpha; 1/2+s z_1)D(\alpha; 1/2+s z_2)\right)\\
&=s^{1+2\alpha}\E\left(\sum_{\ell \geq 2}\frac{(\log \ell)^{\alpha}}{\ell^{1/2+ sz_1}}(\eta_\ell +{\rm i}\theta_\ell)\right)\left(\sum_{j\geq 2}\frac{(\log j)^{\alpha}}{j^{1/2+s z_2}}(\eta_j+{\rm i}\theta_j)\right)\\
&=s
^{1+2\alpha}\sum_{k\geq 2}\frac{(\log k)^{2\alpha}}{k^{1+s(z_1+z_2)}}\me (\eta_k+\ii \theta_k)^2\\
&=(\sigma_1^2-\sigma_2^2+2\ii \rho)s^{1+2\alpha}\sum_{k\geq 2}\frac{(\log k)^{2\alpha}}{k^{1+s(z_1+z_2)
}}\\
&\to \frac{\Gamma(1+2\alpha)(\sigma_1^2-\sigma_2^2+2\ii \rho)}{(z_1+z_2)^{1+2\alpha}},\quad s\to 0+,
\end{align*}
where the convergence is secured by Lemma~\ref{lem:zeta_derivatives} in the Appendix. The right-hand side is equal to $\me \left(\mathcal{I}(\alpha;z_1)\mathcal{I}(\alpha;z_2)\right)$ according to~\eqref{limit1}. The convergence
$$
s^{1+2\alpha}\me \left(D(\alpha; 1/2+s z_1)\overline{D(\alpha; 1/2+s z_2)}\right)~\to~\me \left(\mathcal{I}(\alpha;z_1)\overline{\mathcal{I}(\alpha;z_2)}\right),\quad s\to 0+
$$
follows analogously, by another application of Lemma~\ref{lem:zeta_derivatives}.

\vspace{2mm}
\noindent
{\sc The Lindeberg-Feller sufficient condition.} It suffices to show that, for every fixed $z\in H_0$ and all $\varepsilon>0$,
\begin{equation}\label{eq:lind-feller}
\lim_{s\to 0+}s^{1+2\alpha}\sum_{k\geq 2}\E\left(\left|\frac{(\log k)^{\alpha}}{k^{1/2+sz}}(\eta_k+\ii\theta_k)\right|^2 \1_{\{s^{1+2\alpha}|(\log k)^{\alpha}k^{-1/2-sz}(\eta_k+\ii \theta_k)|>\varepsilon\}}\right)=0.
\end{equation}
Since $ (\log x)^\alpha x^{-1/2} \leq A_\alpha$ for some $A_\alpha>0$ and all $x\geq 2$, we conclude that, for integer $k\geq 2$, $s>0$ and $z\in H_0$,
$$
|(\log k)^\alpha k^{-1/2-sz}|=(\log k)^\alpha k^{-1/2-s\Re(z)}\leq A_\alpha.
$$
Hence, the expression under the limit in~\eqref{eq:lind-feller} is upper bounded by
$$
\left(s^{1+2\alpha}\sum_{k\geq 2}\frac{(\log k)^{2\alpha}}{k^{1+2s
\Re(z)}}\right)\me( \eta^2+\theta^2)\1_{\{\sqrt{\eta^2+\theta^2}>\varepsilon A_\alpha^{-1}s^{-1-2\alpha}\}}.
$$ As a consequence of $\me (\eta^2+\theta^2)<\infty$, the expectation converges to $0$, as $s\to 0+$. According to Lemma~\ref{lem:zeta_derivatives}, the first factor converges to $\Gamma(1+2\alpha)/(2\Re(z))^{1+2\alpha}$. This completes the proof of~\eqref{eq:lind-feller}.

\vspace{2mm}
\noindent
{\sc Tightness.} Let $K$ be an arbitrary compact subset of $H_0$ and $T>0$. In order to prove that the family of distributions of the processes $(s^{1/2+\alpha}D(\alpha; 1/2+s(\cdot)))_{s\in (0,T]}$ is tight on $\mathcal{A}(H_0)$, it suffices to show that
\begin{equation}\label{eq:tightness}
\sup_{s\in (0,T]}\sup_{z\in K}\me |s^{1/2+\alpha}D(\alpha; 1/2+sz)|^2<\infty,
\end{equation}
see the remark after Lemma 2.6 in \cite{Shirai:2012}. Note that
\begin{align*}
\me |D(\alpha; 1/2+sz)|^2&=\me \left( D(\alpha; 1/2+sz)\overline{D(\alpha; 1/2+sz)}\right)\\
&=\E\left(\sum_{\ell\geq 2}\frac{(\log \ell)^{\alpha}}{\ell^{1/2+sz
}}(\eta_\ell+{\rm i}\theta_\ell)\right)\left(\sum_{j\geq 2}\frac{(\log j)^{\alpha}}{j^{1/2+s\overline{z}
}}(\eta_j-{\rm i}\theta_j)\right)\\
&=\E(\eta^2+\theta^2)\sum_{k\geq 2}\frac{(\log k)^{2\alpha}}{k^{1+2s\Re(z)}}
\end{align*}
and that
$$
\sup_{z\in K}\me |s^{1/2+\alpha}D(\alpha; 1/2+sz)|^2\leq \E(\eta^2+\theta^2)s^{1+2\alpha}\sum_{k\geq 2}\frac{(\log k)^{2\alpha}}{k^{1+2sx_0}},$$ where $x_0:=\inf_{z\in K}\Re(z)>0$. Since the right-hand side is bounded in $s
\in (0,T]$ by Lemma~\ref{lem:zeta_derivatives}, inequality \eqref{eq:tightness} follows. The proof of Theorem~\ref{thm:flt_main} is complete.

\begin{rem}
Recall that $\mathcal{D}_{-}^{\alpha}$ denotes the fractional derivative operator defined by ~\eqref{rem:frac_derivative}. In view of~\eqref{eq:frac_dreivative}
$$
s^{1/2+\alpha}D(\alpha; 1/2+s(\cdot))=\mathcal{D}_{-}^{\alpha}\left(s^{1/2}D(0; 1/2+s(\cdot))\right),\quad s>0.
$$
Thus, one possible way to deduce Theorem~\ref{thm:flt_main} could have been to prove it in a simpler situation $\alpha=0$, and then check that $\mathcal{D}_{-}^{\alpha}$ is a.s.~continuous at $\mathcal{I}_0$.
However, we have not been able to obtain appropriate continuity results for $\mathcal{D}_{-}^{\alpha}$, nor locate them in the literature. In view of this we proved Theorem~\ref{thm:flt_main} directly.
\end{rem}

\subsection{Proof of Proposition~\ref{prop:real_zeros}}

Let $\mathcal{A}_{(0,\infty)}(H_0)$ be a subspace of $\mathcal{A}(H_0)$ consisting of all functions $f\in\mathcal{A}(H_0)$ which take real values on $(0,\infty)$. The space $\mathcal{A}_{(0,\infty)}(H_0)$ is endowed with the induced topology. Note that, under the assumptions of Proposition~\ref{prop:real_zeros}, the weak convergence
\begin{equation}\label{eq:real_zeros_proof}
\left(s^{1/2+\alpha} D(\alpha; 1/2+sz)\right)_{z\in H_0}~\Longrightarrow~(\mathcal{I}(\alpha; z))_{z\in H_0},\quad s
\to 0+.
\end{equation}
holds on the space of probability measure on $\mathcal{A}_{(0,\infty)}(H_0)$, since $\mathcal{A}_{(0,\infty)}(H_0)$ is closed in $\mathcal{A}(H_0)$.

Let $A(0,\infty)$ be the set of all $f\in\mathcal{A}_{(0,\infty)}(H_0)$ which do not have multiple real zeros. According to Lemma 4.2 in \cite{Iksanov+Kabluchko+Marynych:2016}, the mapping $\mathcal{A}_{(0,\infty)}(H_0)\ni f\mapsto \Zeros_{(0,\infty)}(f)$ is continuous on $A(0,\infty)$. Thus,
\eqref{eq:real_zeros_pp_conv} follows from~\eqref{eq:real_zeros_proof} provided that
\begin{equation}\label{eq:no_multiple_zeros}
\mathbb{P}\{\mathcal{I}_\alpha \in A(0,\infty)\}=1.
\end{equation}
Recall that $\mathcal{I}_\alpha$ is a centered Gaussian process. Thus, in order to prove~\eqref{eq:no_multiple_zeros} it suffices to check that
\begin{equation}\label{eq:no_multiple_zeros2}
{\rm Var}\,(\mathcal{I}(\alpha;s))>0,\quad s>0,
\end{equation}
see Theorem in \cite{Ylvisaker:1968} or Lemma 4.3 in \cite{Iksanov+Kabluchko+Marynych:2016}. But~\eqref{eq:no_multiple_zeros2} is trivial, since
$$
{\rm Var}\,(\mathcal{I}(\alpha;s))=\frac{\Gamma(1+2\alpha)\sigma_1^2}{(2s
)^{1+2\alpha}}>0,\quad s>0
$$
by~\eqref{limit1}. The proof of Proposition~\ref{prop:real_zeros} is complete.

\section{Proof of Theorem \ref{main}}

We start by proving an intermediate result.
\begin{assertion}\label{lilhalf}
Under the assumptions of Theorem \ref{main},
\begin{equation}\label{lil1}
{\lim\sup}_{s\to 0+}\Big(\frac{s^{1+2\alpha}}{\log\log 1/s}\Big)^{1/2}\sum_{k\geq 2}\frac{(\log k)^\alpha}{k^{1/2+s}}\eta_k \leq \Big(\frac{\sigma_1^2 \Gamma(1+2\alpha)}{2^{2\alpha}}\Big)^{1/2} \quad\text{{\rm a.s.}}
\end{equation}
and
\begin{equation}\label{lil2}
{\lim\inf}_{s\to 0+ }\Big(\frac{s^{1+2\alpha}}{\log\log 1/s}\Big)^{1/2}\sum_{k\geq 2} \frac{(\log k)^\alpha}{k^{1/2+s}} \eta_k\geq -\Big(\frac{\sigma_1^2 \Gamma(1+2\alpha)}{2^{2\alpha}}\Big)^{1/2} \quad\text{{\rm a.s.}}
\end{equation}
\end{assertion}

Replacing $\eta_k$ with $\eta_k/\sigma_1$ we can and do assume that $\sigma_1^2=1$. For $s>0$, put $M(s):=\lfloor 1/s\rfloor$ and, for $s\in (0,1/\eee)$, put $$f_\alpha(s):=\Big(\frac{s^{1+2\alpha}}{c_\alpha \log\log 1/s}\Big)^{1/2},$$ where $$c_\alpha:=\frac{\Gamma(1+2\alpha)}{2^{2\alpha}}.$$ We prove Proposition \ref{lilhalf} via a sequence of lemmas.
\begin{lemma}\label{1}
$\lim_{s\to 0+} f_\alpha(s)\sum_{k=2}^{M(s)}\frac{(\log k)^\alpha}{k^{1/2+s}}\eta_k=0$ {\rm a.s.}
\end{lemma}
\begin{proof}
Put $T_0:=0$, $T_n:=\eta_1+\ldots+\eta_n$ for $n\in\mn$. By the law of the iterated logarithm for standard random walks,
\begin{equation}\label{eq:2}
|T_n|\leq \sup_{k\leq n}|T_k|=O((n \log\log n)^{1/2}),\quad n\to\infty\quad \text{a.s.}
\end{equation}
Hence, there exist $C>0$ and a.s.\ finite $K\geq 3$ such that $|T_n|\leq C(n \log\log n)^{1/2}$ whenever $n\geq K$. Integration by parts yields, for small $s>0$,
\begin{multline*}
\sum_{k=2}^{M(s)}\frac{(\log k)^\alpha}{k^{1/2+s}}\eta_k= \int_{(3/2,\, M(s)]}\frac{(\log x)^\alpha}{x^{1/2+s}}{\rm d}T_{\lfloor x\rfloor}= \frac{(\log M(s))^\alpha T_{M(s)}}{(M(s))^{1/2+s}}-\frac{(\log 3/2)^\alpha \eta_1}{(3/2)^{1/2+s}}\\- \int_{3/2}^K\frac{(\alpha (\log x)^{\alpha-1}-(1/2+s)(\log x)^\alpha)T_{\lfloor x\rfloor}}{x^{3/2+s}}{\rm d}x\\-\int_K^{M(s)}\frac{(\alpha (\log x)^{\alpha-1}-(1/2+s)(\log x)^\alpha)T_{\lfloor x\rfloor}}{x^{3/2+s}}{\rm d}x.
\end{multline*}
Since $\lim_{s\to 0+} (M(s))^s=1$, we infer, as $s\to 0+$,
\begin{multline*}
\frac{(\log M(s))^\alpha |T_{M(s)}|}{(M(s))^{1/2+s}}\sim \frac{(\log M(s))^\alpha |T_{M(s)}|}{(M(s))^{1/2}}=O((\log M(s))^\alpha (\log\log M(s))^{1/2})\\=O((\log 1/s)^\alpha(\log\log 1/s)^{1/2})\quad \text{a.s.}
\end{multline*}
Hence, as $s\to 0+$, $f_\alpha(s)\frac{(\log M(s))^\alpha |T_{M(s)}|}{(M(s))^{1/2+s}}=O(s^{1/2+\alpha}(\log 1/s)^{\alpha})\to 0$ a.s. Further,
\begin{multline*}
\lim_{s\to 0+}\int_{3/2}^K \frac{(\alpha (\log x)^{\alpha-1}-(1/2+s)(\log x)^\alpha) T_{\lfloor x\rfloor}}{x^{3/2+s}}{\rm d}x\\
=\int_{3/2}^K \frac{(\alpha (\log x)^{\alpha-1}-(1/2)(\log x)^\alpha)T_{\lfloor x\rfloor}}{x^{3/2}}{\rm d}x\quad\text{a.s.},
\end{multline*}
and the limit random variable is a.s.\ finite. Finally, by the change of variable $y=s\log x$,
\begin{multline*}
\int_K^{M(s)}\frac{|T_{\lfloor x\rfloor}
|}{x^{3/2+s}}{\rm d}x\leq C\int_K^{M(s)}\frac{(\log\log x)^{1/2}}{x^{1+s}}{\rm d}x=\frac{C}{s}\int_{s\log K}^{s\log M(s)}\eee^{-y}(\log 1/s-\log 1/y)^{1/2}{\rm d}y\\\leq \frac{C(\log 1/s)^{1/2}}{s}\int_0^{s\log M(s)}\eee^{-y}{\rm d}y~\sim~ C (\log 1/s)^{1/2}\log M(s)\sim C(\log 1/s)^{3/2},\quad s\to 0+\quad \text{a.s.},
\end{multline*}
whence, by monotonicity of $x\mapsto (\log x)^\alpha$, $$\int_K^{M(s)}\frac{(\alpha (\log x)^{\alpha-1}+(1/2+s)(\log x)^\alpha)|T_{\lfloor x\rfloor}|}{x^{3/2+s}}{\rm d}x=O((\log 1/s)^{\alpha\vee 0+3/2}),\quad s\to 0+\quad \text{a.s.}$$
Thus, by the triangle inequality, $$\lim_{s\to 0+} f_\alpha(s) \Big|\int_K^{M(s)}\frac{(\alpha (\log x)^{\alpha-1}-(1/2+s)(\log x)^\alpha)T_{\lfloor x\rfloor}}{x^{3/2+s}}{\rm d}x\Big|=0\quad \text{a.s.}$$ The proof of Lemma \ref{1} is complete.
\end{proof}

For $k\in\mn$, $\rho>0$ and $s\in (0,1/\eee)$, define the event $$\mathcal{A}_{k,\,\rho}(s):=\Big\{|\eta_k|>\frac{\rho}{(\log k)^\alpha (\log 1/s)} \Big(\frac{k^{1+s}}{s^{1+2\alpha} \log\log 1/s}\Big)^{1/2}\Big\}.$$
\begin{lemma}\label{3}
For all $\rho>0$,
\begin{equation}\label{aux010}
\lim_{s\to 0+} f_\alpha(s)\sum_{k\geq M(s)+1}\frac{(\log k)^\alpha}{k^{1/2+s}}|\eta_k|\1_{\mathcal{A}_{k,\,\rho}(s)}
=0 \quad\text{{\rm a.s.}}
\end{equation}
and
\begin{equation}\label{aux011}
\lim_{s\to 0+}f_\alpha(s)\sum_{k\geq M(s)+1}\frac{(\log k)^\alpha}{k^{1/2+s}} \me (|\eta_k|\1_{\mathcal{A}_{k,\,\rho}(s)})=0.
\end{equation}
\end{lemma}
\begin{proof}
For fixed $\alpha\geq 0$, the function $x\mapsto x^{-s}(\log x)^\alpha$ attains its overall maximum on $[1,\infty)$ at $\eee^{\alpha/s}$ and
\begin{equation}\label{eq:star}
\frac{(\log x)^\alpha}{x^s} \leq \frac{\alpha^\alpha}{\eee^\alpha s^\alpha}=:\frac{\lambda_\alpha}{s^\alpha},\quad x\geq 1.
\end{equation}
Here, $0^0$ is interpreted as $1$. For $\alpha\in (-1/2, 0)$, we put $\lambda_\alpha:=1$ and note that $$\frac{(\log k)^\alpha}{k^s} \leq 1,\quad k\geq 3.$$ Using this we obtain
\begin{multline*}
\sum_{k\geq M(s)+1}\frac{(\log k)^\alpha}{k^{1/2+s}}|\eta_k|\1_{\mathcal{A}_{k,\,\rho}(s)}\\\leq \lambda_\alpha s^{-\alpha\vee 0} \sum_{k\geq M(s)+1}\frac{|\eta_k|}{k^{1/2}}\1_{\{|\eta_k|k^{-1/2}>\rho \kappa_\alpha^{-1}(s\log\log 1/s)^{-1/2}(\log 1/s)^{-1}\}}=0\quad\text{a.s.}
\end{multline*}
for small enough positive $s$, where $\kappa_\alpha:=2^{\alpha\vee 0} \lambda_\alpha$. The right-hand side vanishes because
$$
\lim_{k\to\infty}k^{-1/2}|\eta_k|=0\quad \text{a.s.}
$$
as a consequence of $\me \eta^2<\infty$ and thereupon $\sup_{k\geq 1}\,|\eta_k|k^{-1/2}<\infty$ a.s. Since the function $s\mapsto (s\log\log 1/s)^{-1/2}(\log 1/s)^{-1}$ is monotone for small $s$ and divergent as $s\to 0+$, the claim is justified by $$\1_{\{|\eta_k|k^{-1/2}>\rho \kappa_\alpha^{-1}(s\log\log 1/s)^{-1/2}(\log 1/s)^{-1}\}}\leq \1_{\{\sup_{k\geq 1}\,|\eta_k|k^{-1/2}>\rho \kappa_\alpha^{-1}(s\log\log 1/s)^{-1/2}(\log 1/s)^{-1}\}}=0\quad \text{a.s.}$$ for small $s$. This completes the proof of \eqref{aux010}.

Relation \eqref{aux011} follows from
\begin{multline*}
\sum_{k\geq M(s)+1}\frac{(\log k)^\alpha}{k^{1/2+s}}\me (|\eta_k|\1_{\mathcal{A}_{k,\,\rho}(s)})\\\leq \lambda_\alpha s^{-\alpha\vee 0}
\sum_{k\geq M(s)+1}k^{-1/2} \me \big(|\eta| \1_{\{ \rho^{-1}\kappa_\alpha (s\log\log 1/s)^{1/2}(\log 1/s)|\eta|>k^{1/2}\}} \big)\\\leq \lambda_\alpha s^{-\alpha\vee 0}\me \left(|\eta| \sum_{k=1}^{\lfloor \rho^{-2}\kappa_\alpha^2 (s\log\log 1/s) (\log 1/s)^2 \eta^2\rfloor} k^{-1/2}\right) \\\leq 2 \rho^{-1} \lambda_\alpha \kappa_\alpha \me \eta^2 s^{-\alpha\vee 0}(s\log\log 1/s)^{1/2} \log 1/s.
\end{multline*}
The proof of Lemma \ref{3} is complete.
\end{proof}

As usual, $\mathcal{A}^c_{k,\,\rho}(s)$ will denote the complement of $\mathcal{A}_{k,\,\rho}(s)$, that is, for $k\in\mn$, $\rho>0$ and $s\in (0,1/\eee)$,
$$\mathcal{A}^c_{k,\,\rho}(s)=\Big\{|\eta_k|\leq \frac{\rho}{(\log k)^\alpha (\log 1/s)} \Big(\frac{k^{1+s}}{s^{1+2\alpha} \log\log 1/s}\Big)^{1/2}\Big\}.$$

\begin{lemma}\label{4}
Fix any $\gamma\in (0, (\sqrt{5}-1)/2)$, pick any $\rho=\rho(\gamma)$ satisfying
\begin{equation}\label{eq:rho_choice}
(1-\gamma)(1+\gamma)^2(2-\exp(4(1+\gamma)\rho c_\alpha^{-1/2}))>1
\end{equation}
and put $s_n:=\exp(-n^{1-\gamma})$ for $n\in\mn$. Then
$$\limsup_{n\to\infty}f_\alpha(s_n)\sum_{k\geq M(s_n)+1}\frac{(\log k)^\alpha\tilde \eta_{k,\rho}(s_n)}{k^{1/2+s_n}}\leq 1+\gamma\quad \text{{\rm a.s.}},$$
where $\tilde \eta_{k,\rho}(s):=\eta_k\1_{\mathcal{A}^c_{k,\,\rho}(s)}-\me (\eta_k\1_{\mathcal{A}^c_{k,\,\rho}(s)})$ for $k\in\mn$ and $s\in (0,1/\eee)$.
\end{lemma}
\begin{proof}
We start by explaining that the stated choice of $\rho$ is indeed possible. Observe that $(1-\gamma)(1+\gamma)^2>1$ whenever $\gamma\in (0,(\sqrt{5}-1)/2)$. Choosing positive $\rho$ sufficiently close to $0$, we can make $2-\exp(4(1+\gamma)\rho c_\alpha^{-1/2})$ as close to $1$ as we wish and particularly ensure that $2-\exp(4(1+\gamma)\rho c_\alpha^{-1/2})>(1-\gamma)^{-1}(1+\gamma)^{-2}$.

Put $$X_\alpha(s):=f_\alpha(s)\sum_{k\geq M(s)+1}\frac{(\log k)^\alpha \tilde \eta_{k,\rho}(s)}{k^{1/2+s}},\quad s\in (0,1/\eee).$$ Using $\eee^x\leq 1+x+(x^2/2)\eee^{|x|}$ for $x\in\mr$ and $\me \tilde \eta_{k,\rho}(s)=0$ we infer, for $u\in\mr$,
\begin{multline*}
\me e^{uX_\alpha(s)}=\prod_{k\geq M(s)+1}\me \exp\Big(uf_\alpha(s)\frac{(\log k)^\alpha \tilde \eta_{k,\rho}(s)}{k^{1/2+s}}\Big)\\\leq \prod_{k\geq M(s)+1} \Big(1+\frac{u^2 (f_\alpha(s))^2}{2}\frac{(\log k)^{2\alpha}}{k^{1+2s}}\me (\tilde \eta_{k,\rho}(s))^2 \exp\Big(|u|f_\alpha(s)\frac{(\log k)^\alpha|\tilde \eta_{k,\rho}(s)|}{k^{1/2+s}}\Big)\Big).
\end{multline*}
Further, the inequality
\begin{equation}\label{eq:mm1}
\begin{split}
|\tilde \eta_{k,\rho}(s)|  \leq |\eta_k|\1_{\mathcal{A}^c_{k,\,\rho}(s)}+\me (|\eta_k|\1_{\mathcal{A}^c_{k,\,\rho}(s)})&\leq \frac{2\rho k^{(1+s)/2}}{(\log k)^\alpha (\log 1/s) (s^{1+2\alpha}\log\log 1/s)^{1/2}}\\ &\leq \frac{2\rho k^{1/2+s}}{(\log k)^\alpha (s^{1+2\alpha}\log\log 1/s)^{1/2}}\quad \text{a.s.},
\end{split}
\end{equation}
which holds true for integer $k\geq 3$ and $s\in (0,1/\eee)$, entails $$\exp\Big(|u|f_\alpha(s)\frac{(\log k)^\alpha|\tilde \eta_{k,\rho}(s)|}{k^{1/2+s}}\Big)\leq \exp\Big(\frac{2\rho c_\alpha^{-1/2} |u|}{\log\log 1/s}\Big)\quad \text{a.s.}$$ This in combination with the inequalities $\me (\tilde \eta_{k,\rho}(s))^2\leq 1$ and $ 1+x\leq \eee^x$ for $x\in\mr$ yields, for $u\in\mr$,
\begin{equation}\label{eq:ss1}
\begin{split}
\me e^{uX_\alpha(s)}&\leq \prod_{k\geq M(s)+1}\exp\Big(\frac{u^2 (f_\alpha(s))^2}{2} \frac{(\log k)^{2 \alpha}}{k^{1+2s}}\exp\Big(\frac{2\rho c_\alpha^{-1/2} |u|}{\log\log 1/s}\Big)\Big)\\& \leq \exp\Big(\frac{u^2}{4\log\log 1/s}\exp\Big(\frac{2\rho c_\alpha^{-1/2} |u|}{\log\log 1/s}\Big)\Big).
\end{split}
\end{equation}
Here, the last inequality is a consequence of $$\sum_{k\geq M(s)+1}
\frac{(\log k)^{2 \alpha}}{k^{1+2s}}\leq \int_1^\infty \frac{(\log x)^{2 \alpha}}{x^{1+2s}}{\rm d}x=\frac{\Gamma(1+2\alpha)}{(2s)^{1+2\alpha}}=\frac{c_\alpha}{2}\frac{1}{s^{1+2\alpha}}.$$ Here, we have used the fact that, for each fixed $s>0$, the function $x\mapsto (\log x)^{2\alpha} x^{-1-2s}$ is decreasing on $(\max (\eee^{2\alpha}, 1), \infty)$.
By the Markov inequality, for $u\geq 0$,
$$\mmp\{X_\alpha(s_n)>1+\gamma\}\leq e^{-(1+\gamma)u}\me e^{uX_\alpha(s_n)}\leq \exp\Big(-(1+\gamma)u+\frac{u^2}{4\log\log 1/s_n}\exp\Big(\frac{2\rho c_\alpha^{-1/2}u}{\log\log 1/s_n}\Big)\Big).$$ Putting $u=2(1+\gamma)\log\log 1/s_n$ we obtain
\begin{multline*}
\mmp\{X_\alpha(s_n)>1+\gamma\}\leq \exp(-(1+\gamma)^2(2-\exp(4(1+\gamma)\rho c_\alpha^{-1/2}))\log\log 1/s_n)\\=\frac{1}{n^{(1-\gamma)(1+\gamma)^2(2-\exp(4(1+\gamma)\rho c_\alpha^{-1/2}))}}.
\end{multline*}
Hence, $\sum_{n\geq 1}\mmp\{X_\alpha(s_n)>1+\gamma\}<\infty$, and an appeal to the Borel-Cantelli lemma completes the proof of Lemma \ref{4}.
\end{proof}

\begin{lemma}\label{5}
Let $\rho=\rho(\gamma)$ and $(s_n)_{n\in\mn}$ be as in Lemma \ref{4} with the only difference that $\gamma\in (0,1/2)$. For $s\in [s_{n+1}, s_n]$, $$\lim_{n\to\infty}f_\alpha(s)\Big(\sum_{k\geq M(s)+1}\frac{(\log k)^\alpha }{k^{1/2+s}}\tilde \eta_{k,\,\rho}(s)-\sum_{k\geq M(s_{n+1})+1} \frac{(\log k)^\alpha}{k^{1/2+s_{n+1}}}\tilde \eta_{k,\,\rho}(s_{n+1})\Big)=0\quad\text{{\rm a.s.}}$$
\end{lemma}
\begin{proof}
Let $s\in [s_{n+1}, s_n]$. Using the fact that $M$ is a nonincreasing function, we write
\begin{multline*}
\sum_{k\geq M(s)+1}\frac{(\log k)^\alpha\tilde \eta_{k,\,\rho}(s)}{k^{1/2+s}}-\sum_{k\geq M(s_{n+1})+1}\frac{(\log k)^\alpha\tilde \eta_{k,\,\rho}(s_{n+1})}{k^{1/2+s_{n+1}}}\\=
\sum_{k=M(s)+1}^{M(s_{n+1})}\frac{(\log k)^\alpha\tilde \eta_{k,\,\rho}(s)}{k^{1/2+s}}+\sum_{k\geq M(s_{n+1})+1}\frac{(\log k)^\alpha(\tilde \eta_{k,\,\rho}(s)-\tilde \eta_{k,\,\rho}(s_{n+1}))}{k^{1/2+s}}\\+ \sum_{k\geq M(s_{n+1})+1}(\log k)^\alpha\Big(\frac{1}{k^{1/2+s}}-\frac{1}{k^{1/2+s_{n+1}}}\Big)\tilde \eta_{k,\,\rho}(s_{n+1})=:I_{n,1}(s)+I_{n,2}(s)+I_{n,3}(s).
\end{multline*}

\noindent {\sc Analysis of $I_{n,1}(s)$}. Recalling that $\me \eta_k = 0$, we further decompose $I_{n,1}$ as follows:
\begin{multline*}
I_{n,1}(s)=\sum_{k=M(s)+1}^{M(s_{n+1})}\frac{(\log k)^\alpha \eta_k}{k^{1/2+s}}+\sum_{k=M(s)+1}^{M(s_{n+1})}\frac{(\log k)^\alpha(- \eta_k\1_{\mathcal{A}_{k,\,\rho}(s)}+\me (\eta_k\1_{\mathcal{A}_{k,\,\rho}(s)}))}{k^{1/2+s}}\\=:I_{n,11}(s)+I_{n,12}(s).
\end{multline*}
We proceed by investigating the summands separately and start with $I_{n,12}(s)$:
$$f_\alpha(s)|I_{n,12}(s)|\leq f_\alpha(s)\sum_{k\geq M(s)+1} \frac{(\log k)^\alpha(|\eta_k|\1_{\mathcal{A}_{k,\,\rho}(s)}+\me (|\eta_k|\1_{\mathcal{A}_{k,\,\rho}(s)}))}{k^{1/2+s}}~\to~ 0,\quad n\to\infty\quad\text{a.s.},$$ where the limit relation is secured by Lemma \ref{3}. Now we pass to $I_{n,11}(s)$. Summation by parts yields
\begin{multline*}
I_{n,11}(s)=\frac{(\log M(s_{n+1}))^\alpha T_{M(s_{n+1})}}{(M(s_{n+1}))^{1/2+s}}-\frac{(\log (M(s)+1))^\alpha T_{M(s)}}{(M(s)+1)^{1/2+s}}\\+
\sum_{k=M(s)+1}^{M(s_{n+1})-1}\Big(\frac{(\log k)^\alpha}{k^{1/2+s}}-\frac{(\log (k+1))^\alpha}{(k+1)^{1/2+s}}\Big)T_k,
\end{multline*}
where, as before, $T_n=\eta_1+\ldots+\eta_n$ for $n\in\mn$. Since $\lim_{n\to\infty}(s_n/s_{n+1})=1$ and the function $f_\alpha$ is regularly varying at $0$ (of index $1/2+\alpha$) we infer
\begin{equation}\label{eq:1}
\lim_{n\to\infty}\frac{f_\alpha(s_n)}{f_\alpha(s_{n+1})}=1.
\end{equation}
The function $f_\alpha$ is increasing on $(0,1/\eee)$. Using this in combination with $\lim_{n\to\infty}(M(s_{n+1}))^{s_{n+1}}=1$, \eqref{eq:2} and \eqref{eq:1} we obtain
\begin{multline*}
f_\alpha (s)\frac{(\log M(s_{n+1}))^\alpha |T_{M(s_{n+1})}|}{(M(s_{n+1}))^{1/2+s}}\leq f_\alpha (s_n)\frac{(\log M(s_{n+1}))^\alpha|T_{M(s_{n+1})}|}{(M(s_{n+1}))^{1/2+s_{n+1}}}\\\sim f_\alpha(s_{n+1})\frac{(\log M(s_{n+1}))^\alpha|T_{M(s_{n+1})}|}{(M(s_{n+1}))^{1/2}}\\=f_\alpha (s_{n+1})O((\log M(s_{n+1}))^\alpha(\log\log M(s_{n+1}))^{1/2})~\to~ 0,\quad n\to\infty \quad\text{a.s.}
\end{multline*}
A similar but simpler argument enables us to conclude that $$\lim_{n\to\infty}\sup_{s\in [s_{n+1},\, s_n]}f_\alpha(s)\frac{(\log (M(s)+1))^\alpha|T_{M(s)}|}{(M(s)+1)^{1/2+s}}=0\quad\text{a.s.}$$ Further, for small enough $s>0$,
\begin{align*}
&\hspace{-1.5cm}f_\alpha (s)\left|\sum_{k=M(s)+1}^{M(s_{n+1})-1}\left(\frac{(\log k)^\alpha}{k^{1/2+s}}-\frac{(\log (k+1))^\alpha}{(k+1)^{1/2+s}}\right)T_k\right|\\
&\leq f_\alpha(s)\sum_{k=M(s)+1}^{M(s_{n+1})-1}\left(\frac{(\log k)^\alpha}{k^{1/2+s}}-\frac{(\log (k+1))^\alpha}{(k+1)^{1/2+s}}\right)|T_k|\\
&\leq f_\alpha (s_n)\left(\sup_{j\leq M(s_{n+1})}|T_j|\right)\sum_{k=M(s)+1}^{M(s_{n+1})-1}\left(\frac{(\log k)^\alpha}{k^{1/2+s}}-\frac{(\log (k+1))^\alpha}{(k+1)^{1/2+s}}\right)\\
&\leq f_\alpha(s_n)\left(\sup_{j\leq M(s_{n+1})}|T_j|\right)\frac{(\log M(s))^\alpha}{(M(s))^{1/2+s}}\\
&\leq f_\alpha(s_n)O((M(s_{n+1})\log\log M(s_{n+1}))^{1/2})\frac{(\log M(s_{n+1}))^\alpha}{M(s_n)^{1/2+s_{n+1}}}\\
&\sim~f_\alpha (s_{n+1})O((\log M(s_{n+1}))^\alpha(\log\log M(s_{n+1}))^{1/2})~\to~0,\quad n\to\infty\quad\text{a.s.}
\end{align*}
We have used \eqref{eq:2} for the last inequality and \eqref{eq:1}, $\lim_{n\to\infty}(M(s_{n+1})/M(s_n))=1$\newline and $\lim_{n\to\infty}(M(s_n))^{s_{n+1}}=1$ for the asymptotic relation. Thus, we have proved that $$\lim_{n\to\infty}\sup_{s\in [s_{n+1},\,s_n]}f_\alpha(s)I_{n,1}(s)=0\quad\text{a.s.}$$

\noindent {\sc Analysis of $I_{n,2}(s)$}. For $n\geq 2$,
\begin{multline*}
f_\alpha(s)|I_{n,2}(s)|\\=f_\alpha(s)\left|\sum_{k\geq M(s_{n+1})+1}\frac{(\log k)^\alpha (\eta_k\1_{\mathcal{A}_{k,\,\rho}(s)}-\me (\eta_k\1_{\mathcal{A}_{k,\,\rho}(s)})-\eta_k\1_{\mathcal{A}_{k,\,\rho}(s_{n+1})}+\me (\eta_k\1_{\mathcal{A}_{k,\,\rho}(s_{n+1})}))}{k^{1/2+s}}\right|\\\leq f_\alpha(s)\sum_{k\geq M(s)+1}\frac{(\log k)^\alpha (|\eta_k|\1_{\mathcal{A}_{k,\,\rho}(s)}+\me (|\eta_k|\1_{\mathcal{A}_{k,\,\rho}(s)}))}{k^{1/2+s}}\\+ \frac{f_\alpha(s_n)}{f_\alpha(s_{n+1})}f_\alpha(s_{n+1})\sum_{k\geq M(s_{n+1})+1}\frac{(\log k)^\alpha(|\eta_k|\1_{\mathcal{A}_{k,\,\rho}(s_{n+1})}+\me (|\eta_k|\1_{\mathcal{A}_{k,\,\rho}(s_{n+1})}))}{k^{1/2+s_{n+1}}}.
\end{multline*}
By Lemma \ref{3} and \eqref{eq:1}, the right-hand side converges to $0$ a.s. as $n\to\infty$.

\noindent {\sc Analysis of $I_{n,3}(s)$}. For $n\in\mn$ and $s\in (0,1/\eee)$, put $$Y_n(s):=\sum_{k\geq M(s_{n+1})+1}\frac{(\log k)^\alpha\tilde \eta_{k,\,\rho}(s_{n+1})}{k^{1/2+s}}.$$ We shall show that, for all $u\in [s_{n+1}, s_n]$,
\begin{equation}\label{eq:3}
\lim_{n\to\infty}f_\alpha(s_n)(Y_n(u)-Y_n(s_{n+1}))=0\quad\text{a.s.}
\end{equation}
We shall use the fact that $Y_n$ is a.s.~continuous and differentiable on $[s_{n+1},s_n]\subset (0,\infty)$ for every fixed $n\in\mathbb{N}$. Indeed, since $\E \eta_{k,\,\rho}(s_{n+1})=0$ and $\E (\eta_{k,\,\rho}(s_{n+1}))^2<\infty$, $Y_n$ is actually analytic on $H_0$ as explained in the introduction.

For $j\in\mn_0$ and $n\in\mn$, put $$F_j(n):=\{t_{j,m}(n):=s_{n+1}+2^{-j}m(s_n-s_{n+1}):0\leq m\leq 2^j\}.$$ Observe that $F_j(n)\subseteq F_{j+1}(n)$ and put $F(n):=\bigcup_{j\geq 0}F_j(n)$. The set $F(n)$ is dense in  $[s_{n+1}, s_n]$. For any $ u\in [s_{n+1}, s_n]$, put
$$
u_j:=\max\{v\in F_j(n): v\leq u\}=s_{n+1}+2^{-j}(s_n-s_{n+1})\left\lfloor\frac{2^j(u-s_{n+1})}{s_n-s_{n+1}}\right\rfloor.
$$
Then $\lim_{j\to\infty}u_j=u$ (we suppress the dependence of $u$ and $u_j$ on $n$ for notational simplicity). An important observation is that either $u_{j-1}=u_j$ or $u_{j-1}=u_j-2^{-j}(s_n-s_{n+1})$. Necessarily, $u_j=t_{j,m}$ for some $0\leq m\leq 2^j$, so that either $u_{j-1}=t_{j,m}$ or $u_{j-1}=t_{j,m-1}$. Since $Y_n$ is a.s.\ continuous on $[s_{n+1}, s_n]$ we obtain
\begin{multline*}
|Y_n(u)-Y_n(s_{n+1})|= \lim_{\ell \to \infty} |Y_n(u_\ell)-Y_n(s_{n+1})|\\
=\lim_{\ell \to \infty}\Big|\sum_{j=1}^\ell (Y_n(u_j)-Y_n(u_{j-1}))+Y_n(u_0)-Y_n(s_{n+1})\Big|\\\leq
\lim_{\ell \to \infty} \sum_{j=0}^\ell\max_{1\leq m\leq 2^j}|Y_n(t_{j,m})-Y_n(t_{j,m-1})|= \sum_{j\geq 0} \max_{1\leq m\leq 2^j}|Y_n(t_{j,m})-Y_n(t_{j,m-1})|.
\end{multline*}
Thus, our purpose is to show that, for all $\varepsilon>0$ and $n_0\in\mn$ large enough, $$\sum_{n\geq n_0}\mmp\Big\{\sum_{j\geq 0}\max_{1\leq m\leq 2^j}f_\alpha(s_n)|Y_n(t_{j,m})-Y_n(t_{j,m-1})|>\varepsilon\Big\}<\infty.$$ Put $a_j:=(j+1) 2^{-j}$ for $j\in\mn_0$. Since $\sum_{j\geq 0}a_j<\infty$ and $\varepsilon>0$ is arbitrary, it is enough to prove that, for all $\varepsilon>0$,
\begin{equation}\label{eq:5}
\sum_{n\geq n_0}\sum_{j\geq 0}\mmp\big\{\max_{1\leq m\leq 2^j}f_\alpha(s_n)|Y_n(t_{j,m})-Y_n(t_{j,m-1})|>\varepsilon a_j \big\}<\infty.
\end{equation}
By the mean value theorem for differentiable functions, there exists $r_{j,m}\in [t_{j,m-1}, t_{j,m}]$ such that
\begin{multline}\label{eq:proof_mean_value_thm}
Y_n(t_{j,m})-Y_n(t_{j, m-1})=-(t_{j,m}-t_{j,m-1})\sum_{k\geq M(s_{n+1})+1}\frac{(\log k)^{1+\alpha}}{k^{1/2+r_{j,m}}}\tilde \eta_{k,\,\rho}(s_{n+1})\\=-2^{-j}(s_n-s_{n+1})\sum_{k\geq M(s_{n+1})+1}\frac{(\log k)^{1+\alpha}}{k^{1/2+r_{j,m}}}\tilde \eta_{k,\,\rho}(s_{n+1}).
\end{multline}
Now we argue as in the proof of Lemma \ref{4} and thus refer to that proof as far as some missing fragments are concerned: for $u\in\mr$ and large $n$,
\begin{align*}
&\hspace{-0.5cm}\me \exp\left(\pm u  \sum_{k\geq M(s_{n+1})+1}\frac{(\log k)^{1+\alpha}}{k^{1/2+r_{j,m}}}\tilde \eta_{k,\,\rho}(s_{n+1})\right)\\
&\leq \prod_{k\geq M(s_{n+1})+1}\left(1+\frac{u^2}{2}\frac{(\log k)^{2+2\alpha}}{k^{1+2r_{j,m}}}\me (\tilde \eta_{k,\,\rho}(s_{n+1}))^2\exp\left(\frac{|u|(\log k)^{1+\alpha}}{k^{1/2+r_{j,m}}}|\tilde \eta_{k,\,\rho}(s_{n+1})|\right)\right)\\
&\leq \prod_{k\geq M(s_{n+1})+1}\left(1+\frac{u^2}{2}\frac{(\log k)^{2+2\alpha}}{k^{1+2s_{n+1}}}\me (\tilde \eta_{k,\,\rho}(s_{n+1}))^2\exp\left(\frac{|u|(\log k)^{1+\alpha}}{k^{1/2+s_{n+1}}}|\tilde \eta_{k,\,\rho}(s_{n+1})|\right)\right).
\end{align*}
We have used $r_{j,m}\geq t_{j, m-1}\geq s_{n+1}$ and monotonicity of the functions involved for the second inequality. Note that (compare with \eqref{eq:star}), for $k \geq M(s_{n+1})+1$,
$$
\frac{\log k}{k^{s_{n+1}/2}}\leq \frac{2}{\eee s_{n+1}}\leq \frac{1}{s_{n+1}}
$$
and thereupon
\begin{align*}
|\tilde \eta_{k,\rho}(s_{n+1})|&\leq |\eta_k|\1_{\mathcal{A}^c_{k,\,\rho}(s_{n+1})}+\me (|\eta_k|\1_{\mathcal{A}^c_{k,\,\rho}(s_{n+1})})\\
&\leq \frac{2\rho k^{(1+s_{n+1})/2}}{(\log k)^\alpha (\log 1/s_{n+1})(s_{n+1}^{1+2\alpha}\log\log 1/s_{n+1})^{1/2}}\\
&\leq \frac{2\rho k^{1/2+s_{n+1}}}{(\log k)^{1+\alpha} (\log 1/s_{n+1})(s_{n+1}^{3+2\alpha}\log\log 1/s_{n+1})^{1/2}}\quad \text{a.s.}
\end{align*}
Thus, using again the inequality $1+x\leq \eee^x$ for $x\in\mathbb{R}$,
\begin{align}
&\hspace{-1cm}\me \exp\left(\pm u  \sum_{k\geq M(s_{n+1})+1}\frac{(\log k)^{1+\alpha}}{k^{1/2+r_{j,m}}}\tilde \eta_{k,\,\rho}(s_{n+1})\right)\notag \\
&\leq \prod_{k\geq M(s_{n+1})+1}\left(1+\frac{u^2}{2}\frac{(\log k)^{2+2\alpha}}{k^{1+2s_{n+1}}}\exp\left(\frac{2\rho |u|}{(\log 1/s_{n+1})(s_{n+1}^{3+2\alpha}\log\log 1/s_{n+1})^{1/2}}\right)\right)\notag \\
&\leq \exp\left(\frac{u^2}{2}\exp\left(\frac{2\rho |u|}{(\log 1/s_{n+1})(s_{n+1}^{3+2\alpha}\log\log 1/s_{n+1})^{1/2}}\right) \sum_{k\geq M(s_{n+1})+1}\frac{(\log k)^{2+2\alpha}}{k^{1+2s_{n+1}}}\right)\notag \\
&\leq \exp\left(\frac{b_\alpha u^2}{2 s_{n+1}^{3+2\alpha}}\exp\left(\frac{2\rho |u|}{(\log 1/s_{n+1})(s_{n+1}^{3+2\alpha}\log\log 1/s_{n+1})^{1/2}}\right)\right),\label{eq:proof_mgf_estimate}
\end{align}
where $b_\alpha:=\Gamma(3+2\alpha) 2^{-3-2\alpha}$. The last inequality follows from $$\sum_{k\geq M(s_{n+1})+1}\frac{(\log k)^{2+2\alpha}}{k^{1+2s_{n+1}}}\leq \int_1^\infty \frac{(\log y)^{2+2\alpha}}{y^{1+2s_{n+1}}}{\rm d}y=\int_0^\infty x^{2+2\alpha} \eee^{-2s_{n+1}x}{\rm d}x=\frac{\Gamma(3+2\alpha)}{(2s_{n+1})^{3+2\alpha}}=\frac{b_\alpha}{s_{n+1}^{3+2\alpha}}.$$
By \eqref{eq:proof_mean_value_thm} and Markov's inequality, for $u\geq 0$ and $n\geq 2$,
\begin{align*}
&\hspace{-0.3cm}\mmp\left\{f_\alpha(s_n)|Y_n(t_{j,m})-Y_n(t_{j,m-1})|>\varepsilon a_j \right\}\\
&=\mmp\left\{\left|\sum_{k\geq M(s_{n+1})+1}\frac{(\log k)^{1+\alpha}}{k^{1/2+r_{j,m}}}\tilde \eta_{k,\,\rho}(s_{n+1})\right|>\frac{\varepsilon(j+1)}{f_\alpha(s_n)(s_n-s_{n+1})}\right\}\\
&\leq \exp\left(-u \frac{\varepsilon(j+1)}{f_\alpha(s_n)(s_n-s_{n+1})}\right)\me \exp \left(u \left|\sum_{k\geq M(s_{n+1})+1}\frac{(\log k)^{1+\alpha}}{k^{1/2+r_{j,m}}}\tilde \eta_{k,\,\rho}(s_{n+1})\right|\right).
\end{align*}
Invoking~\eqref{eq:proof_mgf_estimate} and $\eee^{u|x|}\leq \eee^{ux}+\eee^{-ux}$ for $x\in\mr$, we infer
\begin{multline*}
\mmp\left\{f_\alpha(s_n)|Y_n(t_{j,m})-Y_n(t_{j,m-1})|>\varepsilon a_j \right\}\\
\leq 2\exp\left(-\frac{u\varepsilon(j+1)}{f_\alpha(s_n)(s_n-s_{n+1})}+\frac{b_\alpha u^2}{2 s_{n+1}^{3+2\alpha}}\exp\left(\frac{2\rho |u|}{(\log 1/s_{n+1})(s_{n+1}^{3+2\alpha}\log\log 1/s_{n+1})^{1/2}}\right)\right).
\end{multline*}
Putting
$$
u=\frac{3\varepsilon s_{n+1}^{3+2\alpha}}{2b_\alpha f_\alpha(s_n)(s_n-s_{n+1})},\quad k_n:=\frac{s_{n+1}^{3+2\alpha}}{b_\alpha (f_\alpha(s_n))^2(s_n-s_{n+1})^2}
$$
and
$$
\ell_n:=\frac{s_{n+1}^{3+2\alpha}}{b_\alpha f_\alpha(s_n)(s_n-s_{n+1})(\log 1/s_{n+1})(s_{n+1}^{3+2\alpha}\log\log 1/s_{n+1})^{1/2}}
$$
we obtain, for large $n$,
\begin{multline*}
\mmp\big\{f_\alpha(s_n)|Y_n(t_{j,m})-Y_n(t_{j,m-1})|>\varepsilon a_j \big\}\leq 2\exp(- 3\varepsilon^2 (j+1)k_n/2) \exp\Big(9\varepsilon^2 k_n \exp(3 \rho\varepsilon \ell_n)/8\Big)\\
\leq 2\exp(-3\varepsilon^2 (j+1)k_n/2) \exp\Big(5\varepsilon^2 k_n/4\Big).
\end{multline*}
The last inequality is a consequence of $$\ell_n~\sim~\frac{c_\alpha^{1/2}s_{n+1}}{b_\alpha(s_n-s_{n+1})(\log 1/s_{n+1})}~\sim~\frac{c_\alpha^{1/2}}{b_\alpha (1-\gamma)} n^{2\gamma-1}~\to~0,\quad n\to\infty$$ which ensures that $\exp(3\rho\varepsilon \ell_n)\leq 10/9$ for large enough $n$. Recall that $\gamma\in (0,1/2)$ by assumption.

To proceed, observe that
\begin{equation}\label{eq:4}
k_n~\sim~b_\alpha^{-1}c_\alpha(1-\gamma)^{-1}n^{2\gamma}\log n~\to~\infty,\quad n\to\infty.
\end{equation}
Hence, for large $n$ satisfying $k_n>(2\log 2)/(3\varepsilon^2)$,
\begin{multline*}
\sum_{j\geq 0}\mmp\big\{\max_{1\leq m\leq 2^j}f_\alpha(s_n)|Y_n(t_{j,m})-Y_n(t_{j,m-1})|>\varepsilon a_j \big\}\\
\leq \sum_{j\geq 0}2^j 2\exp(-3\varepsilon^2 (j+1)k_n/2) \exp\Big(5\varepsilon^2 k_n/4\Big)=\frac{2\exp(-\varepsilon^2 k_n/4)}{1-2\exp(-3\varepsilon^2 k_n/2)}.
\end{multline*}
Finally, \eqref{eq:4} entails \eqref{eq:5}.
\end{proof}

We are ready to prove Proposition \ref{lilhalf}.
\begin{proof}[Proof of Proposition \ref{lilhalf}]
We only prove \eqref{lil1}, for \eqref{lil2} is a consequence of \eqref{lil1}  with $-\eta_k$ replacing $\eta_k$. Recall our convention that $\sigma_1^2=1$.

Fix arbitrary $\gamma\in (0,1/2)$ and pick $\rho=\rho(\gamma)$ such that~\eqref{eq:rho_choice} holds. By Lemmas~\ref{4} and~\ref{5}
\begin{equation}\label{eq:proof_of_lilhalf1}
{\lim\sup}_{s\to 0+}f_\alpha(s)\sum_{k\geq M(s)+1}\frac{(\log k)^\alpha}{k^{1/2+s}}\tilde{\eta}_{k,\rho}(s)\leq 1+\gamma\quad\text{a.s.}
\end{equation}
Using relation~\eqref{aux011} and the fact that $\me (\eta_k\1_{\mathcal{A}^c_{k,\,\rho}(s)})=-\me(\eta_k\1_{\mathcal{A}_{k,\,\rho}(s)})$ we conclude that~\eqref{eq:proof_of_lilhalf1} entails
$$
{\lim\sup}_{s\to 0+}f_\alpha(s)\sum_{k\geq M(s)+1}\frac{(\log k)^\alpha}{k^{1/2+s}}\eta_k\1_{\mathcal{A}^c_{k,\,\rho}(s)}\leq 1+\gamma\quad\text{a.s.}
$$
By Lemma~\ref{1} and formula~\eqref{aux010} the latter yields
$$
{\lim\sup}_{s\to 0+}f_\alpha(s)\sum_{k\geq 2}\frac{(\log k)^\alpha}{k^{1/2+s}}\eta_k\leq 1+\gamma\quad\text{a.s.},
$$
which is equivalent to~\eqref{lil1} since the left-hand side does not depend on $\gamma$.
\end{proof}

Here is another intermediate result needed for the proof of Theorem \ref{main}.
\begin{assertion}\label{lilhalf2}
Under the assumptions of Theorem \ref{main},
\begin{equation}\label{lil11}
{\lim\sup}_{s\to 0+}\Big(\frac{s^{1+2\alpha}}{\log\log 1/s}\Big)^{1/2}\sum_{k\geq 2}\frac{(\log k)^\alpha}{k^{1/2+s}}\eta_k \geq
\Big(\frac{\sigma_1^2 \Gamma(1+2\alpha)}{2^{2\alpha}}\Big)^{1/2} \quad\text{{\rm a.s.}}
\end{equation}
and
\begin{equation}\label{lil21}
{\lim\inf}_{s\to 0+ }\Big(\frac{s^{1+2\alpha}}{\log\log 1/s}\Big)^{1/2}\sum_{k\geq 2} \frac{(\log k)^\alpha}{k^{1/2+s}} \eta_k\leq -\Big(\frac{\sigma_1^2 \Gamma(1+2\alpha)}{2^{2\alpha}}\Big)^{1/2} \quad\text{{\rm a.s.}}
\end{equation}
\end{assertion}

Proposition \ref{lilhalf2} will be proved with the help of Lemmas \ref{1} and \ref{3} and two additional lemmas. As in the previous part, without loss of generality, we assume that $\sigma_1^2=1$. We shall also use the sets $\mathcal{A}_{k,\rho}(s)$ and the corresponding truncated variables $\tilde{\eta}_{k,\rho}(s)$ with $\rho=1$.
\begin{lemma}\label{lem:l1}
Fix any $\gamma>0$ and put $\ms_n:=\exp(-n^{1+\gamma})$ for integer $n\geq 2$. Let $N_1$ and $N_2$ be integer-valued, possibly dependent on $\gamma$, functions defined on some (not necessarily the same) right vicinities of $0$ and satisfying $\lim_{s\to 0+}(N_1(s))^s=1$ and $\lim_{s\to 0+}(N_2(s))^s=+\infty$. Then
\begin{equation}\label{eq:inter}
\lim_{n\to\infty} f_\alpha(\ms_n)\sum_{k=2}^{N_1(\ms_n)}\frac {(\log k)^\alpha }{k^{1/2+\ms_n}} \tilde{\eta}_{k,1}(\ms_n)=0\quad\text{{\rm a.s.}}
\end{equation}
and
\begin{equation}\label{eq:inter1}
\lim_{n\to\infty} f_\alpha(\ms_n)\sum_{k\geq N_2(\ms_n)+1}\frac {(\log k)^\alpha }{k^{1/2+\ms_n}} \tilde{\eta}_{k,1}(\ms_n)=0\quad\text{{\rm a.s.}}
\end{equation}
\end{lemma}
\begin{proof}
For $s>0$ close to $0$, put $$Z_{1,\,\alpha}(s):=f_\alpha(s)\sum_{k=2}^{N_1(s)}\frac {(\log k)^\alpha }{k^{1/2+s}} \tilde{\eta}_{k,1}(s).$$ The argument leading to both \eqref{eq:inter} and \eqref{eq:inter1} is similar to that used in the proof of Lemma \ref{4}. In view of this, we provide a proof of \eqref{eq:inter} and only comment on a proof of \eqref{eq:inter1}.

As far as \eqref{eq:inter} is concerned, according to the Borel-Cantelli lemma, it is sufficient to prove that, for all $\varepsilon > 0$,
\begin{equation}\label{eq:w1}
\sum_{n\ge 1} \P\{Z_{1,\,\alpha}(\ms_n)>\varepsilon\} <\infty.
\end{equation}
To this end, we obtain (compare with \eqref{eq:ss1}), for $u\in\mr$,
$$
\E \eee^{u Z_{1,\,\alpha}(s)} \le   \exp \Big( \frac{u^2 (f_\alpha(s))^2}{2}\sum_{k=2}^{N_1(s)}\frac {(\log k)^{2\alpha}}{k^{1+2s}}\exp\Big( \frac{2c_\alpha^{-1/2} |u|}{\log\log 1/s} \Big) \Big).$$
As has already been mentioned, for each fixed $s>0$, the function $x\mapsto (\log x)^{2\alpha}x^{-1-2s}$ is decreasing on $(\eee^{2\alpha\vee 0},\infty)$. Hence,
\begin{multline}
\sum_{k=2}^{N_1(s)}\frac{(\log k)^{2\alpha}}{k^{1+2s}}=O(1)+\sum_{k=\eee^{2\alpha\vee 0}+1}^{N_1(s)} \frac{(\log k)^{2\alpha}}{k^{1+2s}} \le O(1)+\int_1^{N_1(s)+1}\frac{(\log x)^{2\alpha}}{x^{1+2s}}{\rm d}x \\ = O(1)+\frac 1{(2s)^{2\alpha+1}}\int_0^{\log(N_1(s)+1)^{2s}} y^{2\alpha} \eee^{-y}{\rm d}y=o(s^{-1-2\alpha}),\quad s\to 0+. \label{eq:m1}
\end{multline}
The last integral converges because $\alpha>-1/2$ and vanishes as $s\to 0+$ since $\lim_{s\to 0+}(\log(N_1(s)+1))^{2s}=0$ as secured by the assumption. Pick $r>0$ close to $0$ to ensure that, with $\varepsilon$ as in \eqref{eq:w1}, $\delta:=r\varepsilon^{-2}(2c_\alpha)^{-1}\exp(2\varepsilon^{-1}c_\alpha^{-1/2})$ satisfies $(1-\delta)(1+\gamma)>1$. According to the preceding discussion, for small enough $s>0$, $$\sum_{k=2}^{N_1(s)} \frac{(\log k)^{2\alpha}}{k^{1+2s}}\leq \frac{r}{s^{1+2\alpha}}.$$ Summarizing, for $u\in\mr$ and small $s>0$,
\begin{multline*}
\E \eee^{u Z_{1,\,\alpha}(s)} \le   \exp \Big( \frac{u^2 (f_\alpha(s))^2}{2} \frac{r}{s^{1+2\alpha}}\Big)\exp\Big( \frac{2c_\alpha^{-1/2} |u|}{\log\log 1/s} \Big) \Big)\\
=\exp \Big(\frac{r u^2}{2c_\alpha \log\log 1/s} \exp\Big( \frac{2 c_\alpha^{-1/2} |u|}{\log\log 1/s} \Big) \Big).
\end{multline*}
Invoking the Markov inequality with $u = (1/\varepsilon)\log\log(1/s)$ yields, for small $s>0$,
\begin{align*}
\P\{ Z_{1,\,\alpha}(s)>\varepsilon\} &\le \eee^{-u\varepsilon} \E \eee^{uZ_{1,\,\alpha}(s)} \le \exp\big(-\big(1 - r\varepsilon^{-2} (2c_\alpha)^{-1}\eee^{2 \varepsilon^{-1} c_\alpha^{-1/2}}\big) \log\log(1/s)    \big)\\
& = \frac 1{\log(1/s)^{1-\delta}}.
\end{align*}
This entails \eqref{eq:w1}.

The proof of \eqref{eq:inter1} mimics that of \eqref{eq:inter}. One should use the following counterpart of \eqref{eq:m1}:
\begin{equation} \label{eq:mm2}
\sum_{k\geq N_2(s)+1} \frac {(\log k)^{2\alpha}}{k^{1+2s}} \le \int_{N_2(s)}^\infty \frac{(\log x)^{2\alpha}}{x^{1+2s}}{\rm d}x=\frac{1}{(2s)^{1+2\alpha}}\int_{(\log N_2(s))^{2s}}^\infty y^{2\alpha}\eee^{-y}{\rm d}y= o(s^{-1-2\alpha}),\quad s\to 0+.
\end{equation}
\end{proof}
\begin{lemma}\label{lem:new}
Fix sufficiently small $\delta >0$, pick $\gamma >0$ satisfying $(1+\gamma)(1-\delta^2/8)<1$ and put $\ms_n=\exp(-n^{1+\gamma})$ for integer $n\geq 2$. Then
$$
{\lim\sup}_{n\to \infty}f_\alpha(\ms_n) \sum_{k\geq 2}\frac{(\log k)^\alpha}{k^{1/2+\ms_n}}\tilde{\eta}_{k,1}(\ms_n) \geq 1-\delta \quad\text{{\rm a.s.}}
$$
\end{lemma}
\begin{proof}
Let $N_1$ and $N_2$ be integer-valued, possibly dependent on $\gamma$, functions which are nonincreasing in some right vicinities of $0$ and satisfy $\lim_{s\to 0+}(N_1(s))^s=1$, $\lim_{s\to 0+}(N_2(s))^s=+\infty$ and
\begin{equation}\label{choice}
N_1(\ms_{n+1})\ge N_2(\ms_n),\quad n\geq n_0
\end{equation}
for some $n_0\in\mn$. Such a choice is possible. For instance, according to the discussion on p.~7 in \cite{Aymone+Frometa+Misturini:2020}, one can take
$$
N_1(s) = \left\lfloor \left(1+s-\left(\frac{1+\gamma}{\log\log 1/s}\right)^{1/2}\right)^{-1/s}\right\rfloor,\quad s\in (0,s_0),
$$
where $s_0=s_0(\gamma)$ is the smallest positive root of $(1+s)^2\log\log(1/s)=1+\gamma$ on $(0,1/\eee)$, and
$$
N_2(s):=\left\lfloor \left(\frac{\log\log 1/s}{(1+\gamma)\beta_n}\right)^{1/s} \right\rfloor-1, \quad s\in (\ms_{n+1}, \ms_{n}],\quad n\geq 2.
$$
Here,
$$
\beta_n:=\log n\left(1+\ms_{n+1}-(\log(n+1))^{-1/2}\right)^{\exp((n+1)^{1+\gamma}-n^{1+\gamma})}$$ and particularly $\lim_{n\to\infty}\beta_n/\log n=0$.

In view of Lemma \ref{lem:l1}, it is sufficient to check that
\begin{equation}\label{eq:ss9}
\limsup_{n\to\infty}\, Z_{2,\,\alpha}(\ms_n) \ge 1-\delta\quad\text{a.s.},
\end{equation}
where, for small $s>0$, $$Z_{2,\,\alpha}(s)=f_\alpha(s) \sum_{k=N_1(s)+1}^{N_2(s)}\frac {(\log k)^{\alpha}}{k^{1/2+s}}\tilde{\eta}_{k,1}(s).$$ We shall prove that there exists $\overline s\in (0, s_0)$ such that, for all $s\in (0,\overline s)$,
\begin{equation}\label{eq:r3}
\P\{Z_{2,\,\alpha}(s)>1-\delta \}\ge  3^{-1} \eee^{-(1-\delta^2/8) \log\log 1/s}.
\end{equation}
As a consequence, $$\sum_{n\ge n_1}\P\{ Z_{2,\alpha}(\ms_n)>1-\delta  \} \ge 3^{-1} \sum_{n\ge n_1} \frac{1}{n^{(1+\gamma)(1-\delta^2/8)}} = \infty,$$ where $n_1 \ge n_0$ is chosen in such a way that $\ms_n <\overline s$ for $n\geq n_1$. In view of \eqref{choice}, the random variables $Z_{2,\alpha}(\ms_{n_0})$, $Z_{2,\alpha}(\ms_{n_0+1}),\ldots$ are independent. Hence,
divergence of the series entails \eqref{eq:ss9} by the converse part of the Borel-Cantelli lemma.

As a preparation for the proof of \eqref{eq:r3}, consider the event
$$
U_{s,\,\alpha}: = \big\{ 1-\delta < Z_{2,\,\alpha}(s) \le 1 \big\} = \big\{ (1-\delta) g_\alpha(s)< s^{1/2+\alpha} W_\alpha(s) \le g_\alpha(s)\big\},
$$
where $g_\alpha(s): =\big(c_\alpha \log\log 1/s\big)^{1/2}$ and
$$
W_\alpha(s):=\frac{Z_{2,\,\alpha}(s)}{f_\alpha(s)}= \sum_{k =N_1(s)+1}^{N_2(s)} \frac {(\log k)^\alpha}{k^{1/2+s}}\tilde{\eta}_{k,1}(s).
$$
Given $u\in\mr$ and small $s>0$ introduce a new probability measure $\Q_{s,u}$ on $(\Omega,\mathcal{F})$ by the equality
$$
\Q_{s,u}(A) = \frac{\E \left(\eee^{u s^{1/2+\alpha} W_\alpha(s)}\1_{A}\right)}{\E \eee^{u  s^{1/2+\alpha} W_\alpha(s)}}=\frac{\int_{A}\eee^{u s^{1/2+\alpha} W_\alpha(s)}{\rm d}\mmp}{\int_{\Omega}\eee^{u s^{1/2+\alpha} W_\alpha(s)}{\rm d}\mmp},\quad A\in\mathcal{F}.
$$
We suppress the dependence of $\Q_{s,u}$ on $\alpha$ for notational simplicity. Then
\begin{multline}\label{eq:ss10}
\left(\E \eee^{u (s^{1/2+\alpha} W(s)-g_\alpha(s))}\right)\Q_{s,u}(U_{s,\,\alpha})=\eee^{-ug_{\alpha}(s)}\int_{U_{s,\,\alpha}}\eee^{us^{1/2+\alpha}W_{\alpha}(s)}{\rm d}\mmp\\
\leq \eee^{-ug_{\alpha}(s)}\int_{U_{s,\,\alpha}}\eee^{ug_{\alpha}(s)}{\rm d}\mmp=\mmp(U_{s,\,\alpha})\leq \P\{Z_{2,\,\alpha}(s)>1-\delta\}.
\end{multline}
We shall show that, upon choosing an appropriate $u=u(s)=O((\log\log 1/s)^{1/2})$, the expectation on the left-hand side is bounded by $\eee^{-(1-\delta^2/8)\log\log 1/s}$ from below and also prove that $\Q_{s,u}(U_{s,\,\alpha})\geq 1/3$, thereby deriving~\eqref{eq:r3}.

First, we show that, with $u  = O((\log\log (1/s))^{1/2})$,
\begin{equation}\label{eq:p1}
\E  \eee^{u s^{1/2+\alpha} W_\alpha(s)} = \eee^{(c_\alpha/4)u^2 + u^2 h_\alpha(s)}, \quad s\to 0+
\end{equation}
for some function $h_\alpha$ satisfying $\lim_{s\to 0+} h_\alpha(s) = 0$.

Put
$$
\xi_{k}(s):= \frac{s^{1/2+\alpha} (\log k)^\alpha }{k^{1/2 + s}} \tilde{\eta}_{k,1}(s),\quad k\geq 2,\quad s\in (0,1/\eee),
$$
so that, for $u\in\mr$,
\begin{equation}\label{eq:m26}
u s^{1/2+\alpha} W_\alpha(s) =  \sum_{k=N_1(s)+1}^{N_2(s)} u \xi_{k}(s)
\end{equation}
(again, we suppress the dependence of $\xi_{k}(s)$ on $\alpha$). Invoking the second inequality in \eqref{eq:mm1} we obtain
$$
|u \xi_k(s)|\leq \frac{2}{\log 1/s} = o(1), \quad s\to 0+\quad\text{a.s.},
$$
for every $k\geq 3$. Recalling the
asymptotic expansions
$$
\eee^x = 1 + x + x^2/2 + o(x^2)\quad\text{and}\quad\log(1+x)=x+O(x^2),\quad x\to 0,
$$
we infer
\begin{align}\label{eq:ss3}
  \E \eee^{u s^{1/2+\alpha} W_\alpha(s)} &= \prod_{k= N_1(s)+1}^{N_2(s)} \E \exp( u \xi_{k}(s))\notag\\
 & = \prod_{k= N_1(s)+1}^{N_2(s)} \E \big( 1 +  u \xi_{k}(s) + u^2 \xi^2_{k}(s)\big(1/2 + o(1)\big)\big)\notag \\
 & = \exp \sum_{k= N_1(s)+1}^{N_2(s)}\log \big(  1 +   u^2 \E \xi^2_{k}(s)
 \big(1/2 + o(1)\big) \big)\notag\\
 & = \exp \Big(  u^2\big(1/2 + o(1)\big)\sum_{k= N_1(s)+1}^{N_2(s)} \E \xi^2_{k}(s)
 +u^4 O \Big( \sum_{k=N_1(s)+1}^{N_2(s)} (\E \xi^2_{k}(s))^2\Big)\Big).
 \end{align}
By monotonicity,
$$\int_{N_1(s)+1}^{N_2(s)+1} \frac {(\log x)^{2\alpha}}{x^{1+2s}}{\rm d}x \le \sum_{k=N_1(s)+1}^{N_2(s)} \frac{(\log k)^{2\alpha}}{k^{1+2s}} \le \int_{N_1(s)}^{N_2(s)} \frac {(\log x)^{2\alpha}}{x^{1+2s}}{\rm d}x.$$
Combining this with \eqref{eq:m1} and \eqref{eq:mm2} we conclude that $$\lim_{s\to 0+} s^{1+2\alpha} \sum_{k=N_1(s)+1}^{N_2(s)} \frac{(\log k)^{2\alpha}}{k^{1+2s}}=\frac{\Gamma(1+2\alpha)}{2^{1+2\alpha}}=\frac{c_\alpha}{2}.
$$
The latter, together with uniformity in integer $k\in [N_1(s)+1, N_2(s)]$ of the limit relation
$$
\E \tilde{\eta}_{k,1}^2(s) =  \E \eta_k^2 \1_{\mathcal{A}^c_{k,1}(s)} - \big(\E \eta_k \1_{\mathcal{A}^c_{k,1}(s)} \big)^2 ~\to~ 1,\quad s\to 0+,
$$
yields
\begin{equation}\label{eq:ss4}
\sum_{k= N_1(s)+1}^{N_2(s)} \E \xi_k^2(s) = s^{1+2\alpha}\sum_{k=N_1(s)+1}^{N_2(s)} \frac{(\log k)^{2\alpha}}{k^{1+2s}} \; \E \tilde{\eta}_{k,1}^2(s )~\to~\frac{c_\alpha}{2},  \quad s\to 0+.
\end{equation}
Finally,
\begin{multline}\label{eq:ss5}
u^2 \sum_{k= N_1(s)+1}^{N_2(s)} (\E \xi_k^2(s))^2 =u^2
 s^{2+4\alpha} \sum_{k= N_1(s)+1}^{N_2(s)} \frac{(\log k)^{4\alpha}}{k^{2+4s}}(\E \tilde{\eta}_{k,1}^2(s))^2\\
 \leq u^2 s^{2+4\alpha}\sum_{k\geq N_1(s)+1}\frac{(\log k)^{4\alpha}}{k^2} = o(1),\quad s\to 0+.
\end{multline}
Now relations \eqref{eq:ss3}, \eqref{eq:ss4} and \eqref{eq:ss5} entail \eqref{eq:p1}.

Observe that, for any fixed $u\in {\mathbb R}$, formula \eqref{eq:p1} reads $\lim_{s\to 0+}\E \eee^{u s^{1/2+\alpha} W_\alpha(s)}= \E \eee^{(c_\alpha/4)u^2}$, which implies a central limit theorem $s^{1/2+\alpha} W_\alpha(s){\overset{{\rm d}}\longrightarrow}(c_\alpha/2)^{1/2}{\rm Normal} (0,1)$ as $s\to 0+$. Here, as before, ${\rm Normal} (0,1)$ denotes a random variable with the standard normal distribution.

We are ready to prove \eqref{eq:r3}. Put $$u = u(s) = (1-\delta/2)2c_\alpha^{-1/2}(\log\log 1/s)^{1/2}= (1-\delta/2)2c_\alpha^{-1}g_\alpha(s).$$
Formula \eqref{eq:p1} implies that
\begin{equation}\label{eq:p11}
\E \eee^{u (s^{1/2+\alpha} W_\alpha(s)-g_\alpha(s))}=\eee^{-(1-\delta^2/4) \log\log 1/s + o(\log\log 1/s)} \ge \eee^{-(1-\delta^2/8) \log\log 1/s}
\end{equation}
for small $s>0$. Next, we intend to show the $\Q_{s,u}$-distribution of $s^{1/2+\alpha} W_\alpha(s)-((c_\alpha/2)u + 2uh_\alpha(s))$ converges weakly as $s\to 0+$ to the $\mmp$-distribution of $(c_\alpha/2)^{1/2}{\rm Normal} (0,1)$. To this end, we prove convergence of the moment generating functions. Let $\E_{\Q_{s,u}}$ denote the expectation with respect to the probability measure $\Q_{s,u}$. Using \eqref{eq:p1} we obtain, for $t\in\mr$,
\begin{align*}
&\hspace{-0.4cm}\E_{\Q_{s,u}} \eee^{t (s^{1/2+\alpha} W_\alpha(s)-((c_\alpha/2)u + 2 uh(s)))}= \frac{\E \eee^{(t+u) s^{1/2+\alpha} W_\alpha(s)}}{\E \eee^{u s^{1/2+\alpha} W_\alpha(s)}}\eee^{-t((c_\alpha/2)u + 2uh_\alpha(s))}\\
&= \exp\big( (c_\alpha/4)(t+u)^2  + (t+u)^2 h_\alpha(s) - (c_\alpha/4) u^2 -   u^2h_\alpha(s) -   t((c_\alpha/2)u + 2uh_\alpha(s))
 \big)\\
&=\exp\big((c_\alpha/4+h_\alpha(s))t^2\big)\to \exp((c_\alpha/4)t^2)=\me \exp(t(c_\alpha/2)^{1/2}{\rm Normal} (0,1)),\quad s\to 0+.
\end{align*}
The weak convergence ensures that
\begin{multline*}
{\lim\sup}_{s\to 0+} \Q_{s,u} \big\{s^{1/2+\alpha} W_\alpha(s)\leq (1-\delta)g_\alpha(s)\big\}\\
\le \lim_{s\to 0+}\Q_{s,u} \big\{s^{1/2+\alpha} W_\alpha(s)\leq (c_\alpha/2)u + 2uh_\alpha(s) \big\}
=\mmp\{{\rm Normal}\,(0,1)\leq 0\}= 1/2.
\end{multline*}
Since $\lim_{s\to 0+} (g_\alpha(s) - ((c_\alpha/2)u + 2uh_\alpha(s)))=+\infty$, we also have $$\lim_{s\to 0+} \Q_{s,u} \big\{s^{1/2+\alpha} W_\alpha(s)\leq g_\alpha(s)\big\}=1.$$
Summarizing,
\begin{equation}\label{eq:p12}
\Q_{s,u} ( U_{s,\,\alpha})=\Q_{s,u} \big\{s^{1/2+\alpha} W_\alpha(s)\leq g_\alpha(s)\big\}-\Q_{s,u} \big\{s^{1/2+\alpha} W_\alpha(s)\leq (1- \delta)g_\alpha(s)\big\} \ge \frac 13
\end{equation}
for small $s>0$. Now \eqref{eq:r3} follows from \eqref{eq:ss10}, \eqref{eq:p11} and \eqref{eq:p12}. The proof of Lemma \ref{lem:new} is complete.
\end{proof}

\begin{proof}[Proof of Proposition \ref{lilhalf2}]
Again, relation \eqref{lil21} follows from \eqref{lil11} upon replacing $\eta_k$ with $-\eta_k$. To prove \eqref{lil11}, fix sufficiently small $\delta>0$ and pick $\gamma>0$ and $(\ms_n)_{n\in\mn}$ as in Lemma~\ref{lem:new}. Recalling that $\me \eta_k=0$, write
\begin{align*}
f_\alpha(\ms_n)\sum_{k\geq 2}\frac{(\log k)^\alpha}{k^{1/2+\ms_n}}\eta_k&=f_\alpha(\ms_n)\sum_{k=2}^{\lfloor 1/\ms_n\rfloor}\frac{(\log k)^\alpha}{k^{1/2+\ms_n}}\eta_k\\
&-f_\alpha(\ms_n)\sum_{k=2}^{\lfloor 1/\ms_n\rfloor}\frac{(\log k)^\alpha}{k^{1/2+\ms_n}}(\eta_k\1_{\mathcal{A}^c_{k,1}(\ms_n)}-\me (\eta_k\1_{\mathcal{A}^c_{k,1}(\ms_n)}))\\
&+f_\alpha(\ms_n)\sum_{k\geq \lfloor 1/\ms_n\rfloor+1} \frac{(\log k)^\alpha}{k^{1/2+\ms_n}}(\eta_k\1_{\mathcal{A}_{k,1}(\ms_n)}-\me (\eta_k\1_{\mathcal{A}_{k,1}(\ms_n)}))\\&+f_\alpha(\ms_n)\sum_{k\geq 2}\frac{(\log k)^\alpha}{k^{1/2+\ms_n}}(\eta_k\1_{\mathcal{A}^c_{k,1}(\ms_n)}-\me (\eta_k\1_{\mathcal{A}^c_{k,1}(\ms_n)})).
\end{align*}
The first, second and third terms converge to $0$ a.s.\ as $n\to\infty$, by Lemma \ref{1}, formula \eqref{eq:inter} of Lemma \ref{lem:l1}, with $N_1(s)=M(s)=\lfloor 1/s\rfloor$, and Lemma \ref{3}, respectively.
By Lemma \ref{lem:new}, as $n\to\infty$, the upper limit of the fourth term is not smaller than $1-\delta$ a.s. Thus, $${\lim\sup}_{s\to 0+} f_\alpha(s)\sum_{k\geq 2}\frac{(\log k)^\alpha}{k^{1/2+s}}\eta_k \geq {\lim\sup}_{n\to\infty} f_\alpha(\ms_n)\sum_{k\geq 2}\frac{(\log k)^\alpha}{k^{1/2+\ms_n}}\eta_k\geq 1-\delta \quad\text{a.s.},$$ and \eqref{lil11} follows upon letting $\delta$ tend to $0$.
\end{proof}
\begin{proof}[Proof of Theorem \ref{main}]
Formulae \eqref{auxlimsup} and \eqref{auxliminf} follow immediately from Propositions \ref{lilhalf} and \ref{lilhalf2}. It remains to prove \eqref{eq:limset}. To this end, we note that the random function $s\mapsto D(\alpha; 1/2+s)=\sum_{k\geq 2} (\log k)^\alpha k^{-1/2-s} \eta_k$ is a.s.\ continuous on $(0,\infty)$ as the restriction of the random analytic function $z\mapsto D(\alpha; 1/2+z)$, $z\in H_0$.
Therefore, the function $$s\mapsto \Big(\frac{2^{2\alpha}}{\sigma_1^2 \Gamma(1+2\alpha)}\frac{s^{1+2\alpha}}{\log\log 1/s}\Big)^{1/2}D(\alpha; 1/2+s)$$ is a.s.\ continuous on $(0, 1/\eee)$ with ${\lim\sup}_{s\to 0+}=1$ and ${\lim\inf}_{s\to 0+}=-1$. This immediately entails \eqref{eq:limset} with the help of the intermediate value theorem for continuous functions.
\end{proof}

\section{Appendix}

Lemma \ref{lem:zeta_derivatives} is used in the proof of Theorem \ref{thm:flt_main}. While relation \eqref{eq:zeta_derivatives} follows easily if $z$ approaches $0$ along positive reals, it does require a proof if $z$ approaches $0$ along $H_0$.
\begin{lemma}\label{lem:zeta_derivatives}
Let $\beta>-1$ be fixed. Then
\begin{equation}\label{eq:zeta_derivatives}
\lim_{z\to 0,\,z\in H_0}z^{1+\beta}\sum_{k\geq 2}\frac{(\log k)^{\beta}}{k^{1+z}}=\Gamma(1+\beta).
\end{equation}
\end{lemma}
\begin{proof}
Let $\zeta$ be the analytic continuation of the Riemann zeta-function to $\mathbb{C}\setminus\{1\}$. First of all, note that for $\beta\in\mn_0$ the result is a
consequence of the facts, see, for instance, Theorem 12.5 (a) on p.~255 in \cite{Apostol:1976}), that $\lim_{z\to 0,\,z\in \mathbb{C}}z \zeta(1+z)=1$ and $\mc\ni z\mapsto z\zeta(1+z)$ is an entire function. Thus,
$$
\frac{{\rm d}^k}{{\rm d}z^k}(z\zeta(1+z))\Big|_{z=0}\quad  \text{is finite for } k=0,1,\ldots,\beta,
$$
and~\eqref{eq:zeta_derivatives} follows by induction. This argument fails when $\beta$ is not an integer.

Now we provide a proof which works for any $\beta>-1$. To this end, put $z=u+\ii v$ and write
$$
\sum_{k\geq 2}\frac{(\log k)^{\beta}}{k^{1+z}}=\sum_{k\geq 2}\frac{(\log k)^{\beta}\cos(v\log k)}{k^{1+u}} - \ii \sum_{k\geq 2}\frac{(\log k)^{\beta}\sin(v\log k)}{k^{1+u}}.
$$
We are going to apply Euler's summation formula in the form given by Theorem 3 on p.~54 in \cite{Apostol:1976} with $y=1$, $x=N\in\mn$, and
$$
f(t):=t^{-(1+u)}(\log t)^{\beta}\left\{\begin{matrix}\cos\\ \sin \end{matrix}\right\}(v\log t).
$$
It can be checked that, for $u\in [0,1]$ and $|v|\leq 1$,
$$
|f'(t)|=O\left(\frac{1}{t^{3/2}}\right),\quad t\to +\infty,
$$
where the constant in the Landau symbol $O$ does not depend on $u$ and $v$. Hence,
$$
\sum_{k=2}^{N}f(k)=\int_1^N f(t){\rm d}t + \int_1^N (t-\lfloor t\rfloor)f'(t){\rm d}t=\int_1^N f(t){\rm d}t + O(1),
$$
where the symbol $O$ is uniform in $u\in [0,1]$, $|v|\leq 1$ and $N$. Sending $N\to\infty$ yields
$$
\sum_{k\geq 2}f(k)=\int_1^{\infty}f(t){\rm d}t + O(1).
$$
Thus, as $z\to 0$ inside the region $|z|=\sqrt{u^2+v^2}\leq 1$ and $z\in H_0$,
$$
\sum_{k\geq 2}\frac{(\log k)^{\beta}}{k^{1+z}}=\int_1^{\infty}\frac{(\log t)^{\beta}}{t^{1+z}}{\rm d}t+O(1)=\int_0^{\infty}x^{\beta}\eee^{-zx}{\rm d}x+O(1)=z^{-(1+\beta)}\Gamma(1+\beta)+O(1),
$$
and~\eqref{eq:zeta_derivatives} follows.
\end{proof}

\noindent {\bf Acknowledgement}. The research was supported by the High Level Talent Project DL2022174005L of Ministry of Science and Technology of PRC. The authors thank Zakhar Kabluchko for bringing to our attention several references on Gaussian processes and their zeros.


\begin{thebibliography}{99}

\bibitem{Angst+Pham+Poly:2018} J. Angst, V.-H. Pham and G. Poly, \textit{Universality of the nodal length of bivariate random trigonometric polynomials}. Trans. Amer. Math. Soc. \textbf{370} (2018), 8331--8357.

\bibitem{Angst+Poly:2022} J. Angst and G. Poly, \textit{On the zeros of non-analytic random periodic signals}. Int. Math. Res. Not. \textbf{2022} (2022), 4931--4968.

\bibitem{Apostol:1976} T. Apostol, \textit{Introduction to analytic number theory}, Springer, 1976.

\bibitem{Aymone:2019} M. Aymone, \textit{Real zeros of random Dirichlet series}. Electron. Commun. Probab. \textbf{24} (2019), article no. 54, 1--8.

\bibitem{Aymone+Frometa+Misturini:2020} M. Aymone, S. Fr\'{o}meta and R. Misturini, \textit{Law of the iterated logarithm for a random Dirichlet series}. Electron. Commun. Probab. \textbf{25} (2020), article no. 56, 1--14.


\bibitem{Billingsley:1968} P. Billingsley, \textit{Convergence of probability measures}, Wiley, 1968.

\bibitem{Bingham+Goldie+Teugels:1989} N.~H. Bingham, C.~M. Goldie and J.~L. Teugels, \textit{Regular variation}, Cambridge University Press,
1989.

\bibitem{Bovier+Picco:1993} A. Bovier and P. Picco, \textit{A law of the iterated logarithm for random geometric series}. Ann. Probab. \textbf{21} (1993), 168--184.

\bibitem{Ding+Xiao:2006} X. Ding and Y. Xiao, \textit{Natural boundary of random Dirichlet series}. Ukr.~Math.~J. \textbf{58} (2006), 997--1005.

\bibitem{Hough+Krishnapur+Peres+Virag} J. Ben Hough, M. Krishnapur, Y. Peres and B. Vir\`{a}g, \textit{Zeros of Gaussian analytic functions and determinantal point processes},  American Mathematical Society, Vol. \textbf{51}, 2009.

\bibitem{Iksanov+Kabluchko+Marynych:2016} A. Iksanov, Z. Kabluchko and A. Marynych, \textit{Local universality for real roots of random trigonometric polynomials.} Electron. J. Probab. \textbf{21} (2016), article no.~63, 1--19.

\bibitem{Kabluchko:2019} Z. Kabluchko, \textit{An infinite-dimensional helix invariant under spherical projections.} Electron. Commun. Probab. {\bf 24} (2019),  article no.~25, 1--13.

\bibitem{Kabluchko+Klimovsky:2014} Z. Kabluchko and A. Klimovsky, \textit{Complex random energy model: zeros and fluctuations.} Probab. Theor. Relat. Fields. \textbf{158} (2014), 159--196.

\bibitem{Kabluchko+Zaporozhets:2014} Z. Kabluchko and D. Zaporozhets, \textit{Asymptotic distribution of complex zeros of random analytic functions}.  Ann. Probab. \textbf{42} (2014), 1374-1395.

\bibitem{Kahane:1968} J.-P. Kahane, \textit{Some random series of functions}, Cambridge University Press, 1968.

\bibitem{Matsumoto+Shirai:2013} S. Matsumoto and T. Shirai, \textit{Correlation functions for zeros of a Gaussian power series and Pfaffians.} Electron. J. Probab. {\bf 18} (2013), article no.~49, 1--18.

\bibitem{Peres+Virag:2005} Y. Peres and B. Vir\'{a}g, \textit{Zeros of the i.i.d. Gaussian power series: a conformally invariant determinantal process.} Acta Mathematica {\bf 194} (2005), 1--35.

\bibitem{Samko+Kilbas+Marichev:1993} S. Samko, A. Kilbas and O. Marichev, \textit{Fractional integrals and derivatives: theory and applications}, Gordon and Breach Science Publishers, 1993.

\bibitem{Shirai:2012} T. Shirai, \textit{Limit theorems for random analytic functions and their zeros.} RIMS K\^{o}ky\^{u}roku Bessatsu {\bf B34} (2012), 335--359.

\bibitem{Tao+Vu:2015} T. Tao and V. Vu, \textit{Local universality of zeroes of random polynomials.} Int. Math. Res. Not. \textbf{2015} (2015), 5053--5139.

\bibitem{Teicher:1974} H. Teicher, \textit{On the law of iterated logarithm}.  Ann. Probab. \textbf{2} (1974), 714--728.

\bibitem{Titchmarsh:1952} E. Titchmarsh, \textit{The theory of functions}, Oxford University Press, 1952.

\bibitem{Widder:1946} D. Widder, \textit{The Laplace transform}, Princeton Mathematical Series, 1946.

\bibitem{Ylvisaker:1968} N. D. Ylvisaker, \textit{A note on the absence of tangencies in Gaussian sample paths.} Ann. Math. Statist. \textbf{39} (1968), 261--262.

\bigskip
\end{thebibliography}
\end{document}